\documentclass[12pt]{amsart}
\usepackage{a4wide,enumerate,color,graphicx}
\usepackage{amssymb}
\usepackage{color,enumerate}
\usepackage{tikz}

\newcommand{\R}{\mathbb{R}}	
\newcommand{\N}{\mathbb{N}}	
\newcommand{\Z}{\mathbb{Z}}	
\newcommand{\eps}{\varepsilon}	
\newcommand{\pa}{\partial}		
\newcommand{\Div}{\textrm{div}\,}	
\newcommand{\fun}[5]{  			
\begin{array}{cccc}
{#1}\,: & #2 & \to & #3 \\
\phantom{{#1}\,:\,} & #4 & \mapsto & #5
\end{array} }
\newcommand{\na}{\nabla}		

\newcommand{\di}{\displaystyle}  
\newcommand{\utau}{u^{(\tau)}}
\newcommand{\wtau}{w^{(\tau)}}
\newcommand{\Mtau}{M^{(\tau)}}
\newcommand{\cD}{\mathcal{D}}
\newcommand{\Hc}{\mathcal{H}}
\newcommand{\db}{\mathbf{d}}

\newcommand{\charf}{ {\mbox{\large\raisebox{2pt}{$\chi$}}} }

\newtheorem{theorem}{Theorem}   
\newtheorem{lemma}[theorem]{Lemma}   
   
\newtheorem{remark}[theorem]{Remark}

\begin{document}

\title[cross-diffusion system modeling biofilm growth]{Analysis of a degenerate and singular volume-filling cross-diffusion system modeling biofilm growth}

\author[E. S. Daus]{Esther S. Daus}
\address{Institute for Analysis and Scientific Computing, Vienna University of  
Technology, Wiedner Hauptstra\ss e 8--10, 1040 Wien, Austria}
\email{esther.daus@tuwien.ac.at}

\author[J.-P. Mili\v si\'c]{Josipa-Pina Mili\v si\'c}
\address{University of Zagreb, Faculty of Electrical Engineering and Computing, Unska 3, 10000 Zagreb, Croatia}
\email{pina.milisic@fer.hr}

\author[N. Zamponi]{Nicola Zamponi}
\address{Institute for Analysis and Scientific Computing, Vienna University of  
Technology, Wiedner Hauptstra\ss e 8--10, 1040 Wien, Austria}
\email{nicola.zamponi@tuwien.ac.at}

\date{\today}
\thanks{The first and the third author acknowledge partial support from   
the Austrian Science Fund (FWF), grants P22108, P24304, W1245, P27352 and P30000. The second author acknowledges support from the Croatian Science Foundation HRZZ-IP-2013-11-3955. All three authors were partially supported by the bilaterial project No.~HR 04/2018 of the Austrian Exchange Sevice OeAD together with the Ministry of Science and Education of the Republic of Croatia MZO}

\begin{abstract} 
 We analyze the mathematical properties of a multi-species biofilm cross-diffusion model together with very general reaction terms and mixed Dirichlet-Neumann boundary conditions on a bounded domain. This model belongs to the class of volume-filling type cross-diffusion systems which exhibit a porous medium-type degeneracy when the total biomass vanishes as well as a superdiffusion-type singularity when the biomass reaches its maximum cell capacity, which make the analysis extremely challenging. The equations also admit a very interesting non-standard entropy structure. We prove the existence of global-in-time weak solutions, study the asymptotic behavior and the uniqueness of the solutions, and complement the analysis by numerical simulations that illustrate the theoretically obtained results.
\end{abstract}



\maketitle

\section{Introduction}
In this paper we study the mathematical properties of a multi-species cross-diffusion biofilm model recently proposed by Rahman, Sudarsan and Eberl \cite{RaSuEb15}, which describes the local mixing effects between different components of multi-species biofilm colonies. 
These effects are extremely 
useful in waste\-water engineering, where different processes (like aerobic and anoxic 
processes or simultaneous sulfate reduction and nitrogen removal) take place 
simultaneously. It has been pointed out \cite{RaSuEb15} that when two colonies of different 
species merge, spatial biomass gradients can be observed, leading to a spatially 
heterogeneous distribution of biomass. These phenomena can be described by 
cross diffusion, which models how the diffusion of one species is influenced by the 
concentration gradient of the other species in diffusive multi-species systems. Recently, a cross-diffusion biofilm model (see \eqref{cross.diff.sys.1}) was introduced by Rahman, Sudarsan and Eberl \cite{RaSuEb15}, which reflects the same properties as the single-species nonlinear diffusion model \cite{EPL01} (see \eqref{single.species.model}) constructed from experiments, namely a porous-medium type degeneracy when the local biomass vanishes, which leads to a finite speed of propagation of the interface, and a singularity when the biomass reaches the maximum capacity, which guarantees the boundedness of the total mass. It can be formally derived from a space-discrete random-walk lattice model \cite{Ost11, RaSuEb15, ZaJu17} (see Appendix). Due to the cross-diffusion structure, standard techniques like maximum principles and regularity theory cannot be used, and since the diffusion matrix is generally neither symmetric nor positive definite, even the local-in-time existence and boundedness of solutions is hard to prove. However, in recent years significant progress has been made in the analysis of cross-diffusion equations by using the entropy methods. These techniques are based on the identification of a structural condition, namely a formal gradient-flow or entropy structure, allowing for a mathematical treatment, see e.g. \cite{BFPS10,BSW12,DLMT15, Ju15,Ju16,ZaJu17}. 

In this article we prove the global-in-time existence of weak solutions to the multi-species cross-diffusion biofilm model \cite{RaSuEb15}, study its long-time behavior and prove uniqueness of solutions, and we complement our results by some numerical simulations with finite elements by using the free software DUNE \cite{Dune}. For the analytical results, we significantly extend the entropy method in \cite{Ju15}, which is based upon the idea of transforming the system into so called entropy variables such that the diffusion matrix in the new formulation is positive definite. This approach was adapted to a class of degenerate volume-filling type models\footnote{Volume-filling models take into account the fact that concentration may saturate.}  in \cite{ZaJu17}.
However, compared to \cite{ZaJu17}, we have to deal with an additional singularity, which significantly complicates the analysis, but surprisingly also helps to handle very general (even singular) reaction terms. 
The model we study also admits a non-standard entropy structure, which is an interesting mathematical issue by itself; see e.g.~\cite{DDJ17,JMZ18,JZ17} for other works in this direction. 
\medskip

 We are interested in a reaction-cross-diffusion system with volume-filling of the form
\begin{align}\label{cross.diff.sys.00}
 \partial_t u_i - \sum_{j=1}^n\Div \!\!\left(A_{ij}(u)\nabla u_j\right) = r_i(u)\qquad (i=1,\ldots,n)\quad\mbox{on }\Omega,~~t>0. 
\end{align}
Here, $\Omega\subseteq\mathbb{R}^d~ (d\geq1)$ is a bounded domain with Lipschitz boundary, $A(u)=(A_{ij}(u)) \in \mathbb{R}^{n \times n}$ is the diffusion matrix, $u=(u_1,\ldots,u_n):\Omega \times (0,\infty) \to \mathbb{R}^n$ is the vector of proportions of the species within the biofilm, where $u_i=u_i(t,x)$ depends on the time $t$ and the spatial variable $x$ on $\Omega$, and $r=(r_1,\ldots,r_n)$ is the vector of reactions. 
The diffusion coefficients $A_{ij}$ are derived under suitable modeling assumptions in a (formal) diffusive limit from a space-discrete lattice model \cite{Ost11, RSE16, ZaJu17, RaSuEb15} (sketched in the Appendix). They have the form 
\begin{align}\label{diff.coeff}
 A_{ij}(u)= \alpha_i\delta_{ij}p(M)q(M) + \alpha_iu_i(p(M)q'(M) - p'(M)q(M)).
\end{align}
Here, $M=\sum_{i=1}^n u_i$ denotes the total biomass, $\delta_{ij}$ is the Kronecker delta symbol, while the functions $p$ and $q$ measure how favorable it is for species $u_i$ to leave or to arrive at a certain cell in the underlying discrete model.
The constants $\alpha_i>0$ measure how fast biomass moves between neighboring sites of the underlying discrete lattice model (see Appendix).
We point out that eqs.~\eqref{cross.diff.sys.00}--\eqref{diff.coeff} can be also written as
\begin{align}\label{cross.diff.sys.1}
	\partial_t u_i - \alpha_i \Div \!\!\left(p^2(M) \nabla\left(\frac{u_iq(M)}{p(M)}\right)\right) = r_i(u),\quad i=1, \ldots ,n.
\end{align}
For simplicity, we have assumed that $p=p(M)$ and $q=q(M)$ only depend on the total biomass $M$ and are the same for all species $i=1,\ldots,n$. The total biomass $M$ cannot exceed a saturation value (normalized to 1),
which depends on the maximum cell capacity (hence the denomination ``volume filling''); i.e.~$M\leq 1$ must hold during the time evolution of the system.

We supplement the model with initial conditions and mixed Dirichlet-Neumann boundary conditions on the bounded domain $\Omega\subset\mathbb{R}^d~ (d\geq1)$ 
\begin{align}\label{cond.b.0}
u(0,\cdot) & = u^0(\cdot)>0 \quad \mbox{ in } \; 
\Omega,\nonumber\\
u=u_D>0 \quad \mbox{ on } \; \Gamma_D, & \qquad
\nu \cdot\nabla u = 0 \quad \mbox{ on } \; \Gamma_N,\qquad t>0,
\end{align}
where $\partial \Omega=\Gamma_D \cup \Gamma_N$, $\Gamma_D \cap \Gamma_N = \emptyset$ and $|\Gamma_D|>0$. 

For simplicity we have assumed that $u_D = (u_{D,1},\ldots,u_{D,n})\in (0,\infty)^n$ is a constant vector with positive components, but also $x-$dependent boundary data can be treated; see Remark \ref{Rem.Dir} for details. For consistency with the constraint $M\leq 1$, we have to assume
\begin{align}\label{cond.ini.bc}
M_D=\sum_{i=1}^n u_{D,i}<1, \qquad \sup_{x \in \Omega} M_0(x)=\sum_{i=1}^n \sup_{x \in \Omega} u_{i,0}(x) < 1.
\end{align}
These boundary conditions describe well the behavior of the biofilm at its border, but also homogeneous Neumann boundary conditions can be handled, see Remarks \ref{remark.neu.bc}, \ref{Remark.Neumann.LTB}, \ref{Remark.Neumann.uniq}. 

In order to close the model, the functions $p$ and $q$ need to be chosen appropriately. Following the approach in  \cite{RaSuEb15}, the idea is to derive $p$ and $q$ in \eqref{cross.diff.sys.1} from the single-species biofilm model \cite{EPL01, KHE09} (here without reaction) 
\begin{align}\label{single.species.model}
 \pa_t M - \Div\!\!\left(\frac{M^a}{(1-M)^b}\nabla M \right) = 0, \quad a,b >1
\end{align}
under the natural assumption that the evolution of the multi-species model \eqref{cross.diff.sys.1} reduces to the single-species model \eqref{single.species.model} if all species except one vanish or if all species are identical. Thus, we choose $p$ and $q$ such that
\begin{align}
 p^2(M)\nabla\left(\frac{Mq(M)}{p(M)} \right)= \frac{M^a}{(1-M)^b}\nabla M, \quad a,b > 1,
 \label{Eq.1.0}
\end{align}
and consequently
 \begin{align}
  q(M) = \frac{p(M)}{M} \int_0^M \frac{s^a}{(1-s)^b} \frac{ds}{p(s)^2}, \quad M > 0.
 \label{Def.q}
 \end{align}
 For $p$ we assume that
\begin{align}\mbox{$p$ is decreasing,}\quad
p(1)=0,\quad
\exists \kappa, c>0 :~ \lim_{M\to 1} -(1-M)^{1+\kappa}\frac{p'(M)}{p(M)} = c .
  \label{Def.p}
\end{align} 
These hypothesis are consistent with the modeling assumptions of the single-species model \eqref{single.species.model} in \cite{RSE16}.
The last assumption quantifies how fast $p$ decreases to $0$ for $M\to 1$; in fact, an integration yields the bound
\begin{align}
p(M)\leq C_1\exp(-C_2(1-M)^{-\kappa})\qquad 0<M<1.\label{est.p.1}
\end{align}
This hypothesis on $p$ is not assumed in \cite{RaSuEb15}, and is needed here only for technical reasons in Lemma~\ref{bound.nabla}. 
\medskip

As mentioned before, due to the cross-diffusion structure, standard techniques like maximum principles and regularity theory cannot be used, but still, in recent years lots of progress has been made in the analysis of cross-diffusion equations by identifying a formal gradient-flow or entropy structure \cite{Ju15,Ju16}. Following the approach therein, we assume that there  exists a convex function $h: \mathcal{D} \to \Omega$ called {\em entropy density} with $\mathcal{D}\subseteq \mathbb{R}^n$ such that the matrix $B=A(u)h''(u)^{-1}$ is positive semi-definite, and \eqref{cross.diff.sys.00} can be written as
\begin{align*}
 \pa_t u - \Div\left(B\nabla h'(u)\right) = r(u),
\end{align*}
where $h'$ and $h''$ are the Jacobian and the Hessian of $h$, respectively. This structural assumption has two very useful consequences. 
First, $H[u]=\int_{\Omega}h(u)\,dx$ is a Lyapunov functional along solutions to \eqref{cross.diff.sys.00} and \eqref{diff.coeff} if the reaction term
vanishes, because
\begin{align*}
 \frac{dH}{dt}[u(t)]=\int_{\Omega}h'(u)\cdot\pa_tu\,dx = -\int_{\Omega}\nabla u : h''(u)A(u)\nabla u\,dx = - \int_{\Omega}\nabla w : B\nabla w \,dx \leq 0,
\end{align*}
were $w=h'(u)$ are so called {\em entropy variables}. This often yields gradient-type estimates for $u$ if suitable lower bounds for the matrix $h''(u)A(u)$ are known. Second, if $h'$ is invertible on $\mathcal{D}$, then it holds that $u=(h')^{-1}(w) \in \mathcal{D}$. Consequently we get that if $\mathcal{D}$ is a bounded domain, then we obtain lower and upper bound for $u$ without using a maximum principle. In our case, we require that any solution $u$ to \eqref{cross.diff.sys.00}--\eqref{cond.b.0}
takes values in the set
\begin{equation}\label{Def.D}
\cD = \left\{u\in (0,\infty)^n ~:~ \sum_{i=1}^n u_i < 1\right\}. 
\end{equation}
Moreover, we define the (relative) entropy functional $H[u]$ of the system \eqref{cross.diff.sys.00} as
\begin{equation}\label{H}
H[u] = \int_\Omega h^*(u|u_D) dx , \quad h^*(u|u_D) = h(u) - h(u_D) - h'(u_D)\cdot (u-u_D),
\end{equation}
where the entropy density $h(u)$ is given by
 \begin{align}
  h(u) = \sum_{i=1}^n \left(u_i \log u_i - u_i + 1 \right) 
         + \int_0^M \log \left( \frac{q(s)}{p(s)} \right)\,ds,
 \label{Entropy}
 \end{align}
and thus the entropy variables for $i=1,\ldots,n$ read as
\begin{align}\label{entr.var.2}
w_i&=\frac{\pa}{\pa u_i}h^*(u|u_D) = \frac{\pa h}{\pa u_i}(u)-\frac{\pa h}{\pa u_i}(u_D)
=\log\left(\frac{u_i q(M)}{p(M)}\right)-\log\left(\frac{u_{D,i} q(M_D)}{p(M_D)}\right).
\end{align}
 We can show that the entropy dissipation leads to the following very interesting degenerate-singular entropy estimate
\begin{align*}
 \int_0^T \int_{\Omega} \frac{M^{a-1}|\na M|^2}{(1-M)^{1+b+\kappa}}\,dxdt + \sum_{i=1}^n \int_0^T \int_{\Omega} p(M)q(M)|\na \sqrt{u_i}|^2\,dxdt \leq C\qquad t>0,\nonumber
\end{align*}
which has the same kind of singular-degenerate type structure like the nonlinear diffusion coefficients $M^a/(1-M)^b$ in the single species model \eqref{single.species.model}, i.e. a degeneracy when $M\to 0$ and a singularity when $M\to 1$.
Moreover, it leads to the following uniform bound for the singularity:
\begin{align}\label{bound.sing.int.0}
\int_0^T\int_{\Omega}(1-M)^{1-b-\kappa}dx dt \leq C.
\end{align}
This estimate is essential in order to control the nonlinear terms in the equations; 
furthermore, it implies that $M<1$ a.e.~in $Q_T$, i.e.~saturation of the biofilm is excluded.

\subsection{State of the art}
In the literature, several model classes for biofilms can be found. The first class
consists of deterministic continuous equations based on the one-dimensional
Wanner-Gujer model \cite{WG86}. 
An important assumption of this model is that the volume fractions
occupied by the different species add up to unity. However, no mixing of initially 
separated species can occur under this assumption, which contradicts results of
microscopic experiments, where spatially heterogeneous distributions of biomass 
could be observed. A second model class are stochastic discrete multi-species biofilm
models, which do not need this problematic assumption for the volume fractions, and
the amount of mixing can be decided by the user by formulating local interaction rules.
However, these models have the drawback that mixing is often overemphasized and that the
numerical solution is generally very time-consuming.

In order to compensate the disadvantages of the model classes described above, 
Rahman, Sudarsan, and Eberl \cite{RaSuEb15} introduced a two-species diffusion model
which captures the quantitative amount of local mixing effects between the species 
within a biofilm colony with the help of cross-diffusion terms. The derivations of this system from mass balances or
from discrete lattice models do not (a priori) impose the condition that the
volume fractions must add up to unity. Besides the new cross-diffusion effects, it has two additional difficulties: (i) a porous medium degeneracy when the total biomass $M=0$ is vanishing; (ii)
a super-diffusion singularity when the total biomass equals one $M=1$. Due to property (i) the interface between the aqueous phase and biofilm region propagates with a finite speed. Property (ii) ensures that the solutions of \eqref{cross.diff.sys.00} are bounded by the maximum cell density $M\leq 1$. Moreover, we are able to prove that similarly to the results in \cite{EEZ02, EZE09} it holds that $M<1$ a.e. in $Q_T$ for all $t>0$ in the case of mixed Dirichlet-Neumann boundary conditions
(even for nonzero Dirichlet data; see Step 4 in the proof of Theorem \ref{Thm.Existence} for details). On the other hand, in the case of homogeneous Neumann boundary conditions on the whole $\partial\Omega$ we need to make sure that the total mass $\mathcal{M}(t)=|\Omega|^{-1}\int_{\Omega}M(x,t)dx$ remains strictly smaller than one for any time in order to prevent blowup, see Remark \ref{remark.neu.bc}, which is again similar to the results in \cite{EEZ02, EZE09}.
While an existence analysis for the single-species biofilm model is available in \cite{EZE09}, the mathematical analysis of the multi-species model \eqref{cross.diff.sys.1}--\eqref{cond.b.0} has not been carried out so far (up to our knowledge). We note that Laoshen Li studied traveling wave solutions and instability conditions of a reaction-cross-diffusion biofilm model in \cite{Li14}, and recently Schulz and Knabner analyzed an effective model for biofilm growth in \cite{SK17, SK2017}. Our technique is based on the boundedness-by-entropy method of A.~J\"{u}ngel (see \cite[Theorem 3]{Ju15} and \cite{Ju16}), which was refined for a general class of degenerate volume-filling type cross-diffusion models in \cite{ZaJu17}. However, the model considered in this article does not only exhibit a degeneracy at $M=0$, but also a very interesting singularity at $M=1$, which goes far beyond the framework of \cite{ZaJu17}. We point out that the mentioned singularity significantly complicates the analysis, but on the other hand it also yields an additional a priori bound of $(1-M)$ to a negative power (see \eqref{bound.sing.int.0}), which allows to handle very general (even singular) reaction terms (see assumptions \eqref{hp.r1}--\eqref{hp.r4} for the reaction terms). Note that in \cite{ZaJu17} no reaction terms were treated, and only no-flux boundary conditions were considered.

\subsection{Mathematical assumptions on the reaction terms}\label{Sec:model}
We assume that 
\begin{equation}\label{hp.r1}
r(u) = r^{D}(u) + \tilde r(u) ,
\end{equation}
with $r^D, \tilde r$ continuous in the set $\{ M<1\}$ satisfying the following conditions:
\begin{align}
\label{hp.r2}
& \exists\lambda_r\geq 0 : ~~
\sum_{j=1}^n r_j^D(u)\left(\frac{\pa h}{\pa u_j}(u)-\frac{\pa h}{\pa u_j}(u_D)\right)\leq \lambda_r\Big(1 + h^*(u|u_D)\Big),\\
\label{hp.r3}
& \exists C_r\geq 0, ~ 0\leq \mu < b-1, ~ s>0: ~~ |\tilde r_i(u)|\leq \frac{C_r u_i^s}{(1-M)^\mu}\quad (i=1,\ldots,n),\\
\label{hp.r4}
& \exists C_r'\geq 0, ~ 0\leq \eta < b+\kappa-1 : ~~ |r_i^D(u)|\leq \frac{C_r'}{(1-M)^\eta}\quad (i=1,\ldots,n).
\end{align}
Note that $r(u)$ decomposes into a ``dissipative'' part $r^D$ (i.e. a part which can be controlled by the entropy density, see \cite[assumption $(H3)$]{Ju15}), and a remainder $\tilde r$
which can be controlled by means of the entropy dissipation (see proof of Lemma \ref{lem.dei}). 
Let us point out that these reaction terms are rather general, in fact, even singular reaction terms are allowed.

\subsection{Structure of the paper.}
The main results are given in Section~\ref{Sec.Main}.
Proofs of some auxiliary results, like the asymptotic behavior of $p$ and $q$, the convexity of $h$, the invertibility of $h'$ and the lower bound for the entropy dissipation are given in Section~\ref{Sec.Aux}. 
Sections~\ref{Sec.Exis}, \ref{Sec.LTB} and \ref{Sec.Uniq} are devoted to proofs of the existence, long-time behavior and the uniqueness result, respectively. 
Finally, in Section~\ref{sec.appendix} we discuss the formal derivation of the model and the underlying modeling assumptions (Subsection \ref{Sec.Formal.der}),
we show some numerical simulations (Subsection~\ref{Sec.Num}), and 
we prove a non-standard version of the Poincar\'{e} inequality used within this paper (Subsection~\ref{appendix: aux.res}).
\section{Main results}\label{Sec.Main}
The first result we prove is about the global-in-time existence of weak solutions to \eqref{cross.diff.sys.1}--\eqref{cond.b.0}.
In the following, $Q_T \equiv \Omega\times (0,T)$ for every $T>0$. 
\begin{theorem}[Existence theorem]\label{Thm.Existence}
Under the assumptions 
\eqref{Def.q}, \eqref{Def.p}, \eqref{hp.r1}--\eqref{hp.r4},
eqs.~\eqref{cross.diff.sys.1}--\eqref{cond.b.0} have a solution $u : \Omega\times (0,\infty)\to\R^n$ such that,
for every $T>0$ and $1\leq i\leq n$,
\begin{align*}
& u_i\geq 0,~~ M=\sum_{j=1}^n u_j < 1\quad\mbox{a.e. in }Q_T,\\
& M^{\frac{a+1}{2}}\in L^2(0,T; H^1(\Omega)),\quad M^{\frac{a+1}{2}}\na u_i\in L^2(Q_T),\\
& (1-M)^{1-\kappa}\in L^\infty(0,T; L^1(\Omega)),\quad (1-M)^{1-b-\kappa}\in L^1(Q_T),\\
& \pa_t u_i\in (L^\frac{\rho+1}{\rho}(0,T; W^{1,\frac{\rho+1}{\rho}}(\Omega))\cap L^{\frac{b+\kappa-1}{b+\kappa-1-\eta}}(Q_T))',
\end{align*}
where $a$, $b$, $\kappa$, $\eta$ are as in \eqref{Def.q}, \eqref{Def.p}, \eqref{hp.r4}, and $\rho = \min\{1,\kappa/(b-1)\}$.
Moreover, for any $t>0$ the following entropy inequality holds:
\begin{align}\label{entr.inequ.2}
  &\frac{d}{dt}\int_{\Omega}h^*(u|u_D)\,dx + 2\sum_{i=1}^n\alpha_i\int_\Omega p(M)^2\left|\nabla\sqrt{\frac{u_i q(M)}{p(M)}}\right|^2 dx\\
   &\qquad\leq C_1\int_{\Omega}h^*(u|u_D)\,dx + C_2,\nonumber
 \end{align}
where $C_1$ and $C_2$ are some suitable nonnegative constants. Finally, $C_1=C_2=0$ if $\lambda_r=C_r=0$ in \eqref{hp.r2} {and} \eqref{hp.r3}.
\end{theorem}
The proof of Theorem~\ref{Thm.Existence} is based upon the semi-discretization in time of \eqref{cross.diff.sys.00}. 
The resulting elliptic problem reads as
$$ \frac{u_i^j-u_i^{j-1}}{\tau} + \Div \!\!\left(p(M^j)^2\na\left(\frac{u_i^j q(M^j)}{p(M^j)}\right)\right) = r_i(u^j),\quad i=1,\ldots,n,~~
x\in\Omega . $$
A higher order regularizing term is also added, which is needed in order to prove the well-posedness of the time-discretized equations.
The key tool in the analysis is a discrete entropy inequality:
\begin{align*}
  &\frac{1}{\tau}\int_{\Omega}(h^*(u^j|u_D)-h^*(u^{j-1}|u_D))dx 
  + 2\sum_{i=1}^n\alpha_i\int_\Omega p(M^j)^2\left|\nabla\sqrt{\frac{u_i^j q(M^j)}{p(M^j)}}\right|^2 dx\\
   &\qquad\leq C_1\int_{\Omega}h^*(u^j|u_D)\,dx + C_2,\nonumber
 \end{align*}
which yields crucial gradient estimates for the solution $u^j$ to the time-discretized problem.
The entropy dissipation satisfies the bound (see Lemma \ref{bound.nabla})
\begin{align*}
  &\int_0^T \int_{\Omega} p^2(M)\left|\na \sqrt{\frac{u_iq(M)}{p(M)}} \right|^2\,dxdt\\
  &\qquad \geq C \int_0^T \int_{\Omega} \frac{M^{a-1}|\na M|^2}{(1-M)^{1+b+\kappa}}\,dxdt + \sum_{i=1}^n \int_0^T \int_{\Omega} p(M)q(M)|\na \sqrt{u_i}|^2\,dxdt, 
\end{align*}
which leads to \eqref{bound.sing.int.0}. 
This estimate of the singularity is crucial to control the nonlinear terms, 
and implies that no saturation occurs in the biofilm.

The second result we prove concerns the long-time behavior of the solutions  to \eqref{cross.diff.sys.00}--\eqref{cond.b.0}.
\begin{theorem}[Convergence to steady state] \label{Thm.LTB}
Let all the assumptions from Theorem~\ref{Thm.Existence} be fulfilled. In addition, assume that $\lambda_r=C_r=0$ in \eqref{hp.r2} and \eqref{hp.r3} and $b\geq 2$.
Then there exists a constant $C>0$ such that for any $t>0$ it holds
\[  \sum_{i=1}^n \|u_i(t) - u_{D,i} \|^2_{L^2(\Omega)} \leq \frac{C}{1+t}. \]
\end{theorem}
This means that the solutions to \eqref{cross.diff.sys.1}--\eqref{cond.b.0} converge to the constant steady state $u_D$ as $t\to\infty$.
The main idea of the large-time asymptotic analysis of $u_i(t) : = u_i(\cdot,t)$
is to exploit the entropy inequality \eqref{entr.inequ.2} in the case when $C_1=C_2=0$.
We show that the entropy dissipation dominates the square of the entropy functional, i.e.
\begin{align*}
\sum_{i=1}^n\alpha_i\int_\Omega p(M)^2\left|\nabla\sqrt{\frac{u_i q(M)}{p(M)}}\right|^2 dx\geq C\left( \int_{\Omega}h^*(u|u_D)\,dx \right)^2 ,
\end{align*} 
 from where we deduce that the convergence is of order $1/t$, as $t \to \infty$. 
 Finally, strict convexity of the relative entropy density (see Lemma~\ref{pos.def}), gives the convergence in $L^2$-norm.
We note that we were not able to prove an exponential decay rate due to 
the lack of suitable convex Sobolev inequalities for \eqref{cross.diff.sys.00}.
However, we point out that our numerical simulations suggest that exponential decay should hold, see Subsection \ref{Sec.Num}.

 The third result we present is about the uniqueness of the solution to \eqref{cross.diff.sys.00}--\eqref{cond.b.0}.
 Uniqueness of solutions is achieved provided that additional assumptions on the reaction term are made. Precisely, we assume that
 functions $r^{(0)}_1,\ldots,r^{(0)}_n, r^{(1)}, R : [0,1)\to\R$ exist such that
 \begin{align}
   & r_i(u) = r^{(0)}_{i}(M) + r^{(1)}(M) u_i,\quad i=1,\ldots,n,\quad u\in\mathcal D,   \label{assum.react.0}\\
   & \exists\eps_0>0 : ~ r^{(0)}_i(M) \geq \max\{0,\eps_0 r^{(1)}(M)\},\quad i=1,\ldots,n,\quad M \in [0,1), \label{assum.react.0b}\\
   & \exists C_R\in\R : ~ \sum_{j=1}^n r_j(u) = R(M) + C_R M,\quad u\in\mathcal D,\label{assum.react.1}\\
   & \exists C_R'>0: ~ \frac{|R(M)|}{M} + |R'(M)| \leq C_R' M^{a/2}, \quad M \in (0,1).  \label{assum.react.2}
  \end{align} 
An example of reaction term satisfying both sets of assumptions \eqref{hp.r1}--\eqref{hp.r4}, \eqref{assum.react.0}--\eqref{assum.react.2} is
\begin{equation}
 r_i(u) = u_{D,i} - u_i, \qquad i=1,\ldots,n,\quad u\in\mathcal D. \label{example.r}
\end{equation}
Furthermore, we assume that the parameters $\alpha_1,\ldots,\alpha_n$ are all the same (and therefore without loss of generality we set $\alpha_i=1$, $i=1,\ldots,n$).

\begin{theorem}[Uniqueness of solutions] \label{Thm.Uniqueness} 
 Let the assumptions of Theorem~\ref{Thm.Existence} hold. Furthermore, we assume that $\alpha_i=1$ for all $i=1,\ldots,n$ and that the reaction terms satisfy the assumptions given by \eqref{assum.react.0}--\eqref{assum.react.2}. 
 Then there exists a unique weak solution to \eqref{cross.diff.sys.1}--\eqref{cond.b.0}.
\end{theorem} 
 We point out that the proof of uniqueness of weak solutions for strongly coupled cross-diffusion systems is delicate. 
 Similarly like in \cite{ChJu18, ZaJu17}, our uniqueness proof is based on a combination of the $H^{-1}$ method and the technique of Gajewski \cite{Gaj94}.

\section{Auxiliary results}\label{Sec.Aux}

 In this section we state technical results which are used for proving the main results of this paper: asymptotic behavior of functions $p(M)$ and $q(M)$ when $M \to 0$ and $M \to 1$, the convexity of the entropy density $h$, the invertibility of the gradient of the relative entropy density $h^*$ with respect to the variable $u$ and finally the lower bound for the entropy dissipation.

\begin{lemma}[Asymptotic behavior of $p$, $q$]
 Let $a,b > 1$, $\kappa > 0$ and $0 < M < 1$. For functions $p$ and $q$ defined by \eqref{Def.q} and \eqref{Def.p}, there exist positive constants $C_1$, $C_2$ and $C_3$ such that
 \begin{align}
 \lim_{M\to 1} \frac{p(M)q(M)}{(1-M)^{1+\kappa-b}} &= C_1 ,  \label{est.mto1}\smallskip\\
 \lim_{M \to 1} \frac{\log\big(q(M)/p(M)\big)}{ (1-M)^{-\kappa}} &= C_2, \label{esti.4}\smallskip\\
\lim_{M\to 0}M^{-a}q(M) &=\frac{C_3}{p(0)}. \label{asy.q.Mto0}
\end{align}
\end{lemma}
\begin{proof}
 The proof of limits given by formulas \eqref{est.mto1} and 
 \eqref{esti.4} directly follows using l'H\^{o}pital's rule.
 In order to show \eqref{asy.q.Mto0} we perform the change of variable $s = M\sigma$ in the integral appearing in the definition of $q$ \eqref{Def.q}. 
\end{proof}
\begin{lemma}[Convexity of $h$]\label{pos.def}
 It holds that the matrix $\di \Hc(u)\equiv\left(\frac{\pa w_i}{\pa u_j}(u)\right)_{i,j=1}^ n$ is positive definite and symmetric on $\mathcal{D}$,
 where $w_i$ is defined in \eqref{entr.var.2} and $\cD$ is given by \eqref{Def.D}.
\end{lemma}
\begin{proof}
Direct calculation using \eqref{Def.q} gives
\begin{align}\label{H.ij}
 \Hc_{ij}=\frac{\partial w_i}{\partial u_j} = \left(\frac{\delta_{ij}}{u_i} -\frac{1}{M}\right) + \frac{M^a(1-M)^{-b} p(M)^{-2}}{\int_0^M s^a (1-s)^{-b} p(s)^{-2}\,ds}.
\end{align}
Next, let us write the matrix $\Hc$ as the sum
$ \Hc = \mathcal{A} + \mathcal{B}$, with $\mathcal{A}$, $\mathcal{B} \in \mathbb{R}^{n\times n}$
given by 
\[ \mathcal{A}\equiv \mbox{diag}(\frac{1}{u_1},\ldots,\frac{1}{u_n})-\frac{1}{M}\mathcal{C} \quad \textrm{ and } \quad 
\mathcal{B}\equiv\frac{M^a(1-M)^{-b} p(M)^{-2}}{\int_0^M s^a (1-s)^{-b} p(s)^{-2}\,ds}\mathcal{C},\]
and with $\mathcal{C} \in \mathbb{R}^{n\times n}$, $\mathcal{C}_{ij}:=1$ for all $i,j \in \{1,\ldots,n\}$.
Clearly, the matrix $\mathcal{B}$ is positive semidefinite, since for any $v \in \R^n$ one has
\begin{align*}
 v\cdot \mathcal{B}v = \frac{M^a(1-M)^{-b} p(M)^{-2}}{\int_0^M s^a (1-s)^{-b} p(s)^{-2}\,ds} \left(\sum_{i=1}^n v_i\right)^2 \geq 0.
\end{align*}
On the other side, matrix $\mathcal{A}$ is also positive semidefinite. Namely, for any $v \in \R^n$ we have
\begin{align*}
 v\cdot \mathcal{A}v = \sum_{i=1}^n \frac{v_i^2}{u_i} -\frac{1}{M}\left(\sum_{i=1}^n v_i \right)^2 \geq \sum_{i=1}^nz_i^2 -\frac{1}{\sum_{j=1}^n u_j} \left(\sum_{i=1}^n u_i \right)\left(\sum_{i=1}^n z_i^2\right) = 0. 
\end{align*}
where we used the notation $z_i := v_i/\sqrt{u_i}$ and the Cauchy-Schwarz inequality. 
Consequently, it follows that
$\Hc=\mathcal{A}+\mathcal{B}$ is positive semidefinite. It remains to show the strict positive definiteness of $\Hc$. For this, we take a vector  $v \in \mathbb{R}^n$ and show that if $v \cdot \Hc v = 0$, then it follows that $v_i= 0$ for $i=1,\ldots, n$. Let be $v \cdot \Hc v = 0$, then since matrices
$\mathcal{A}$ and $\mathcal{B}$ are positive semidefinite, it holds $v\cdot \mathcal{B}v=0$, and
$v\cdot \mathcal{A}v=0$. Now, from  $v\cdot \mathcal{B}v=0$ follows directly that 
 $\sum_{i=1}^n v_i = 0$. On the other side, from
\begin{align*}
 0 = v\cdot \mathcal{A}v = \sum_{i=1}^n \frac{v_i^2}{u_i} - \frac{1}{M}\left(\sum_{i=1}^n v_i\right)^2 = \sum_{i=1}^n \frac{v_i^2}{u_i},
\end{align*}
we get directly  that $v_i=0$ for all $i=1,\ldots, n$. Therefore $\Hc$ is positive definite in $\mathcal{D}$.
\end{proof}
\begin{lemma}[Invertibility of $(h^*)'$] 
 The function $(h^*)': \mathcal{D} \to \mathbb{R}^n$ is invertible, where $ \mathcal{D}$ is defined in \eqref{Def.D}.
\end{lemma}
\begin{proof} 
First, note that due to \eqref{entr.var.2} we have (slight change)
\begin{align}
 \frac{u_{D,i}}{M_D} e^{w_i} = \frac{u_i}{M_D}\frac{q(M)/p(M)}{q(M_D)/p(M_D)}.
 \label{L6:1}
  \end{align}
Now we define the auxiliary function
\begin{align}\label{aux.func}
 \Phi(M):=\frac{Mq(M)}{p(M)}. 
\end{align}
After summing the relation \eqref{L6:1} for $i=1,\ldots n$, one gets 
\begin{align}
\Phi(M)=\Phi(M_D)\sum_{i=1}^n \frac{u_{D,i}}{M_D}e^{w_i}. 
 \label{eq.Phi}
 \end{align}
Note that function $\Phi(M)$ is strictly increasing with $\Phi(0)=0$ and $\lim_{M\to 1}\Phi(M)=+\infty$.
 Thus, there exists a unique solution $M=M[w] \in (0,1)$ to the nonlinear equation \eqref{eq.Phi}.
Replacing $M=M[w]$ into relation \eqref{L6:1} and then solving the resulting equation for $u_i$ yields the statement.
\end{proof}
\begin{lemma}[Lower bound for the entropy dissipation]\label{bound.nabla} For any sufficiently smooth function $u : Q_T\to\cD$ it holds that 
 \begin{align*}
  &\int_0^T \int_{\Omega} p^2(M)\left|\na \sqrt{\frac{u_iq(M)}{p(M)}} \right|^2\,dxdt\\
  &\qquad \geq C \int_0^T \int_{\Omega} \frac{M^{a-1}|\na M|^2}{(1-M)^{1+b+\kappa}}\, dxdt 
  + \sum_{i=1}^n \int_0^T \int_{\Omega} p(M)q(M)|\na \sqrt{u_i}|^2\,dxdt, 
 \end{align*}
 where $\kappa>0$ is defined in \eqref{Def.p}.
\end{lemma}
\begin{proof}
 Let us define $f(M) := \sqrt{q(M)/p(M)}$. Direct calculation gives:
\begin{align*}
p^2(M)\sum_{i=1}^n \left|\nabla\sqrt{\frac{u_i q(M)}{p(M)}}\right|^2 
=p(M)q(M) \sum_{i=1}^n|\na\sqrt u_i|^2 + p(M)^2f'(M)\left(M f'(M) + f(M) \right)|\na M|^2 .
\end{align*}
Let us first show that for $0<M<1$ the function $f'(M)$ is strictly positive.    
Note that   
\begin{align*}
 2\frac{f'(M)}{f(M)} 
 = \frac{d}{d M}\log \Big( \frac{q(M)}{p(M)} \Big). 
 \end{align*}  
 Using the definition \eqref{Def.q} we have
 \begin{align}
 2\frac{f'(M)}{f(M)} 
=\frac{M^a (1-M)^{-b}p^{-2}(M)}{\int_0^M s^a (1-s)^{-b}p^{-2}(s)ds} - \frac{1}{M}.
\label{L7.eq.1}
 \end{align}
Since $p$ is decreasing it holds $(1-s)^{-b}p^{-2}(s)\leq (1-M)^{-b}p^{-2}(M)$ for $0\leq s\leq M$. Therefore
\begin{align*}
2\frac{f'(M)}{f(M)} &\geq \frac{M^a }{\int_0^M s^a ds} - \frac{1}{M} = \frac{a}{M} > 0, \quad 0<M<1,
\end{align*}
from where it follows that $f'(M)>0$ for $0<M<1$.
Using this result, we get
\begin{align*}
p^2(M) f'(M)\left(M f'(M) + f(M) \right)\geq p^2(M) M (f'(M))^2 \geq \frac{a^2 p^2(M) }{4M}(f(M))^2 = \frac{a^2 q(M)p(M)}{4M}.
\end{align*}
Since $p(s)\leq p(0)$ and $p(0) > 0$ for $s>0$, one has
\begin{align*}
p(M)q(M)  = \frac{p^2(M)}{M} \int_0^M \frac{s^a}{(1-s)^b} \frac{ds}{p(s)^2}\geq \frac{p^2(M)}{p^2(0)}\frac{1}{M}\int_0^M  \frac{s^a}{(1-s)^b} ds. 
\end{align*}
Since $p(M)\geq p(1/2)$ for $0\leq M \leq 1/2$ and $(1-s)^{-b} \geq 1$ we have
\begin{align*}
 p(M)q(M)  \geq \frac{p^2(1/2)}{p^2(0)}\frac{M^a}{a+1}.
\end{align*}
Therefore we get
\begin{equation}
p^2(M) f'(M)\left(M f'(M) + f(M) \right)\geq C M^{a-1},\qquad 0\leq M\leq \frac{1}{2} .
\label{est.left}
\end{equation}
On the other side, let us find the lower bound of term $ p^2(M) f'(M)\left(M f'(M) + f(M) \right) $  for $1/2 \leq M  < 1$.
For that purpose, we can make the following estimate:
\begin{align*}
 p(M)^2 f'(M) &\left(M f'(M) + f(M) \right) \geq \frac12 p^2(M) (f'(M))^2 
    = \frac12 p(M)q(M)\left(\frac{f'(M)}{f(M)}\right)^2.
\end{align*}
It remains to bound the term $f'(M)/f(M)$ from below for $1/2 \leq M  < 1$.
For the moment, we go back to \eqref{L7.eq.1}. Note that for $s\leq M$ it holds $s^a\leq M^a$, so the following estimate holds:
\begin{align*}
2\frac{f'(M)}{f(M)}\geq \frac{(1-M)^{-b}p^{-2}(M)}{\int_0^M (1-s)^{-b}p^{-2}(s)ds} - \frac{1}{M}.
\end{align*}
We want to find a lower bound for the right-hand side of the above inequality for $M$ close to 1. 
Thanks to \eqref{est.p.1} we have
$$ \lim_{M\to 1}(1-M)^{1+\kappa-b}p^{-2}(M)= +\infty.$$
Applying the l'H\^{o}pital's rule and using \eqref{Def.p}, one gets
\begin{align*}
\lim_{M\to 1}\frac{(1-M)^{1+\kappa-b}p^{-2}(M)}{\int_0^M (1-s)^{-b}p^{-2}(s)ds}
=\lim_{M\to 1}\left( (1+\kappa-b)(1-M)^\kappa - 2(1-M)^{1+\kappa}\frac{p'(M)}{p(M)} \right) = 2c .
\end{align*}
It follows that there exists a constant $c_1>0$ such that 
$$ \frac{f'(M)}{f(M)}\geq c_1(1-M)^{-(1+\kappa)}, \quad \textrm{ for }\;\frac{1}{2}\leq M < 1.$$ 
From the above estimate and \eqref{est.mto1} we deduce 
\begin{align}
 p(M)^2 f'(M)  \left(M f'(M) + f(M) \right) 
  \geq C (1-M)^{-1-b-\kappa},\qquad \frac{1}{2}\leq M < 1.
     \label{est.right}
\end{align}
Putting \eqref{est.left}, \eqref{est.right} together yields that there exists a constant $C>0$ such that
\begin{align*}
 	p(M)^2f'(M)\left(Mf'(M) + f(M)\right) \geq \frac{C M^{a-1}}{(1-M)^{1+b+\kappa}}, \quad 0\leq M <1, 
     \end{align*}   
which finishes the proof of this Lemma.
\end{proof}
\section{Proof of Theorem~\ref{Thm.Existence}}\label{Sec.Exis}
For $m\in\Z$, $m\geq 1$ we define the space 
$$ H^m_D(\Omega) = \left\{ u\in H^m(\Omega)~:~ u\equiv 0~\mbox{on }\Gamma_D \right\}. $$
The proof is divided into several steps.\medskip\\
{\bf Step 1: discretization.}
Fix $T>0$. For $N\in \N$ we define $\tau = T/N$, $t_j = \tau j$ ($j=0,\ldots,N$), $u_{i}^{0} = u_{i,0}$ $(i=1,\ldots,n)$. 
In order to have a compact embedding $H^m(\Omega)\hookrightarrow L^\infty(\Omega)$ we choose $m$ to be the smallest integer such that $m>d/2$. 
For $j\geq 1$ consider the problem: 
\begin{align}\nonumber
&\mbox{given $w^{j-1}\in H_D^m(\Omega)$, find $w^j\in H_D^m(\Omega)$ such that}\\
\label{d.r.eq}
&\sum_{i=1}^n\int_\Omega\left( \frac{u_i^j - u_i^{j-1}}{\tau}\phi_i
+ \alpha_i p(M^j)^2\nabla\left(\frac{u_i^j q(M^j)}{p(M^j)}\right)\cdot\nabla\phi_i - r_i(u^j)\phi_i \right)dx \\
&\qquad = - \tau \sum_{i=1}^n(w_i^j , \phi_i)_{H^m(\Omega)}, \qquad 
\forall \phi=(\phi_1,\ldots,\phi_n)\in H^m_D(\Omega)^n,\nonumber
\end{align}
where $u^{j-1}, u^j : \Omega\times (0,T)\to\R^n$ are defined by 
$$h'(u^{j-1})-h'(u_D)=w^{j-1},\qquad h'(u^{j})-h'(u_D)=w^{j} , $$
and $M^j\equiv\sum_{i=1}^n u_i^j$, while
$(\cdot,\cdot)_{H^m(\Omega)}$ denotes the standard scalar product in $H^m(\Omega)$.
We point out that, since $h' : \cD\to\R^n$ is invertible, then $u^{j-1}$, $u^j$ are well defined.
\medskip\\
{\bf Step 2: fixed point.} We solve \eqref{d.r.eq} via Leray-Schauder fixed point theorem. Let us define the mapping
$$\fun{F}{L^\infty(\Omega)\times [0,1]}{L^\infty(\Omega)}{(w^*,\sigma)}{w}$$
where $w$ is the solution of the linearized approximated problem
\begin{align}
\label{d.r.l.eq}
&\tau \sum_{i=1}^n (w_i, \phi_i)_{H^m(\Omega)} \\
\nonumber
& = -\sigma\sum_{i=1}^n\int_\Omega\left( \frac{u_i^* - u_i^{j-1}}{\tau}\phi_i
+ \alpha_i p(M^*)^2\nabla\left(\frac{u_i^* q(M^*)}{p(M^*)}\right)\cdot\nabla\phi_i - r_i(u^*)\phi_i \right)dx \\
&\qquad \forall\phi\in H_D^m(\Omega; \R^n),\nonumber
\end{align}
with $u^* : \Omega\times (0,T)\to\R^n$ defined by $h'(u^{*})-h'(u_D)=w^{*}$ and $M^*\equiv\sum_{i=1}^n u_i^*$.

We first point out that $F$ is well defined. In fact, assumption $m>d/2$ implies that $H^m(\Omega)\hookrightarrow L^\infty(\Omega)$.
Since $w^*\in H^m(\Omega; \R^n)$, this means that $\inf_\Omega M^* > 0$, $\sup_\Omega M^* < 1$, and $u_i^*, M^*\in H^m(\Omega)$.
These properties ensure that the right-hand side of \eqref{d.r.l.eq} defines a continuous linear functional $f : \phi\in H^m(\Omega; \R^n)\mapsto f(\phi)\in\R$.
Therefore we can deduce by Lax-Milgram Lemma the existence of a unique solution $w\in H_D^m(\Omega; \R^n)\subset L^\infty(\Omega; \R^n)$ to \eqref{d.r.l.eq}.

Next, we observe that $F(\cdot,0)\equiv 0$ (trivial). Choosing $\phi=w$ in \eqref{d.r.l.eq} allows us to easily deduce that $\|w\|_{H^m(\Omega)}\leq C$
for some constant $C=C[w^*, u^{j-1}]>0$. This bound and the compact embedding $H^m_D(\Omega)\hookrightarrow L^\infty(\Omega)$ imply that $F$
is compact. By standard arguments we can prove that $F$ is continuous. 

Let us now assume that $w\in H^m(\Omega; \R^n)$ is a fixed point of $F(\cdot,\sigma)$ for some $\sigma\in [0,1]$, and rename $u\equiv u^*$, $M\equiv M^*$ for better readability.
Define $q_D\equiv q(\sum_{i=1}^n u_{D,i})$, $p_D\equiv p(\sum_{i=1}^n u_{D,i})$. By choosing $\phi_i=w_i$ in \eqref{d.r.l.eq} and exploiting \eqref{entr.var.2}
we obtain
\begin{align*}
\sigma\sum_{i=1}^n\int_\Omega\left( \frac{u_i - u_i^{j-1}}{\tau}w_i
+ 4\alpha_i p(M)^2\left|\nabla\sqrt{\frac{u_i q(M)}{p(M)}}\right|^2 - r_i(u)w_i \right)dx 
+\tau\sum_{i=1}^n\|w_i\|_{H^m(\Omega)}^2 = 0 .
\end{align*}
However, since $w = h'(u) - h'(u_D)$ and $h$ is convex, it follows
$$ \sum_{i=1}^n (u_i-u_i^{j-1}) w_i \geq h^*(u|u_D) - h^*(u^{j-1}|u_D), $$
and therefore
\begin{align*}
\frac{\sigma}{\tau}\int_\Omega h^*(u|u_D)dx 
&+ 4\sigma\sum_{i=1}^n\alpha_i\int_\Omega p(M)^2\left|\nabla\sqrt{\frac{u_i q(M)}{p(M)}}\right|^2 dx \\
& +\tau\sum_{i=1}^n\|w_i\|_{H^m(\Omega)}^2 \leq \frac{\sigma}{\tau}\int_\Omega h^*(u^{j-1}|u_D)dx+ \sigma\sum_{i=1}^n\int_\Omega  r_i(u)w_i dx .
\end{align*}
By applying Lemma~\ref{bound.nabla} we deduce
\begin{align}\label{d.e.i.1}
& \frac{\sigma}{\tau}\int_\Omega h^*(u|u_D)dx + 2\sigma\sum_{i=1}^n\alpha_i\int_\Omega p(M)^2\left|\nabla\sqrt{\frac{u_i q(M)}{p(M)}}\right|^2 dx
+\tau\sum_{i=1}^n\|w_i\|_{H^m(\Omega)}^2\\ 
&\qquad\leq \frac{\sigma}{\tau}\int_\Omega h^*(u^{j-1}|u_D)dx+\sigma\sum_{i=1}^n\int_\Omega  r_i(u)w_i dx
-C_0\sigma \int_{\Omega}\frac{M^{a-1}|\na M|^2}{(1-M)^{1+b+\kappa}}\,dx .\nonumber
\end{align}
We are going to show that the right-hand side of the above inequality can be bound by the entropy.
Let $M_D = \sum_{i=1}^n u_{D,i}\in (0,1)$. It holds that
\begin{align*}
&\int_0^T \int_{\Omega}\frac{M^{a-1}}{(1-M)^{1+b+\kappa}}|\na M|^2\,dxdt \\
&\qquad\geq \int_{\{M\geq M_D\}}\frac{M^{a-1}}{(1-M)^{1+b+\kappa}}|\na M|^2\,dxdt \\
&\qquad\geq M_D^{a-1}\int_0^T \int_{\Omega}\frac{\charf_{\{M \geq M_D\}}}{(1-M)^{1+b+\kappa}}|\na M|^2\,dxdt\\
&\qquad = C\int_0^T \int_{\Omega}\Big|\na\Big( (1-M)^{\frac{1-b-\kappa}{2}}-(1-M_D)^{\frac{1-b-\kappa}{2}}\Big)_+ \Big|^2\,dxdt\\
&\qquad\geq C_P\int_0^T \int_{\Omega}  \Big( (1-M)^{\frac{1-b-\kappa}{2}}-(1-M_D)^{\frac{1-b-\kappa}{2}}\Big)_+^2  \,dxdt,
 \end{align*}
where we used the Poincar\'{e} inequality in the last line. Thus we obtain
\begin{align}\label{bound.sing}
&\int_0^T \int_{\Omega}\frac{M^{a-1}|\na M|^2}{(1-M)^{1+b+\kappa}}\,dxdt \geq 
c\int_0^T \int_{\Omega}(1-M)^{1-b-\kappa}dx dt - C .
 \end{align}
Thanks to \eqref{hp.r1}--\eqref{hp.r3} the reaction term can be bounded as follows:
\begin{align*}
 &\sum_{i=1}^n \int_{\Omega} r_i(u) w_i dx \leq 
 \lambda_r\int_\Omega (1+h^*(u\vert u_D))dx \\
 &\qquad + C_r \sum_{i=1}^n \int_\Omega \frac{u_i^s}{(1-M)^\mu}
\left|\log\left(\frac{u_i q(M)}{p(M)} \right) - \log\left(\frac{u_{D,i} q_D}{p_D} \right) \right|dx \\
&\leq  \lambda_r\int_\Omega (1+h^*(u\vert u_D))dx + C_r c_1\int_\Omega (1-M)^{-\mu}dx \\
&\qquad + C_r c_2\int_\Omega\frac{M^s |\log(q(M)/p(M))|}{(1-M)^\mu}dx .
\end{align*}
By using \eqref{esti.4} we deduce
\begin{align*}
&\sum_{i=1}^n \int_{\Omega} r_i(u) w_i dx \leq \lambda_r\int_\Omega (1+h^*(u\vert u_D))dx 
+ C_r C\int_\Omega (1-M)^{-\kappa-\mu}dx .
\end{align*}
Due to assumption \eqref{hp.r3} we have $\mu<b-1$, so we can apply Young inequality to the right-hand side of the above estimate and conclude
\begin{align*}
&\sum_{i=1}^n \int_{\Omega} r_i(u) w_i dx \leq \lambda_r\int_\Omega (1+h(u))dx 
+ C_r\varepsilon \int_{\Omega}(1-M)^{1-b-\kappa}\,dx + C_r C(\varepsilon).
\end{align*}
By choosing $\eps>0$ small enough in the above estimate and exploiting \eqref{bound.sing}, from \eqref{d.e.i.1} we get
 \begin{align}\label{d.e.i.fix}
\frac{\sigma}{\tau}\int_\Omega h^*(u|u_D)dx 
&+ 2\sigma\sum_{i=1}^n \alpha_i \int_\Omega p(M)^2\left|\nabla\sqrt{\frac{u_i q(M)}{p(M)}}\right|^2 dx \\
& +\tau\sum_{i=1}^n\|w_i\|_{H^m(\Omega)}^2 \leq \left(\frac{\sigma}{\tau}+C_1\right)\int_\Omega h^*(u^{j-1}|u_D)dx + C_2 \nonumber
\end{align}
for some suitable constants $C_1, C_2\geq 0$, which are independent of both $\sigma$ and $\tau$.
Moreover, the constants $C_1, C_2$ can be chosen to be equal to zero in the case that $\lambda_r = C_r = 0$ in \eqref{hp.r2}, \eqref{hp.r3}.
In particular, \eqref{d.e.i.fix} yields a $\sigma-$uniform bound for $w$ in $H^m(\Omega)$, and a fortiori in $L^\infty(\Omega)$.

Thanks to Leray-Schauder's fixed point theorem we infer the existence of a fixed point $w^j\in H^m(\Omega;\R^n)$ for $F(\cdot,1)$, that is, a solution to \eqref{d.r.eq}.
\medskip\\
{\bf Step 3: uniform in $\tau$ a-priori estimates.} Let us define the piecewise constant-in-time functions
$$ \utau(t) = u^0\charf_{\{0\}}(t) + \sum_{j=1}^N u^j\charf_{(t_{j-1},t_j]}(t),\qquad
\wtau(t) = w^0\charf_{\{0\}}(t) + \sum_{j=1}^N w^j\charf_{(t_{j-1},t_j]}(t), $$
and let $\Mtau = \sum_{i=1}^n \utau_i$. We also define the discrete backward time derivative operator $D_\tau$ as follows:
for every function $f : Q_T\to \R$,
$$ D_\tau f(x,t) = \frac{f(x,t)-f(x,t-\tau)}{\tau}\qquad x\in\Omega,\quad t\in [\tau, T] . $$
Now \eqref{d.e.i.fix} can be written in the form of the discrete entropy inequality given by following:
\begin{lemma}[Discrete entropy inequality]\label{lem.dei} For all $t\in [0,T]$ it holds
 \begin{align}\label{d.e.i}
D_\tau\int_\Omega h^*(\utau|u_D)dx &+ 2\sum_{i=1}^n \alpha_i \int_\Omega p(\Mtau)^2\left|\nabla\sqrt{\frac{\utau_i q(\Mtau)}{p(\Mtau)}}\right|^2 dx \\
& +\tau\sum_{i=1}^n\|\wtau_i\|_{H^m(\Omega)}^2 \leq C_1\int_\Omega h^*(\utau|u_D)dx + C_2, \nonumber
\end{align}
for some suitable constants $C_1, C_2\geq 0$.
Moreover, the constants $C_1, C_2$ can be chosen to be equal to zero in the case that $\lambda_r = C_r = 0$ in \eqref{hp.r2} and \eqref{hp.r3}.
\end{lemma}
From Lemma~\ref{bound.nabla}, the entropy inequality \eqref{d.e.i} and estimate \eqref{bound.sing} 
we deduce (via a discrete Gronwall argument) the following 
bounds, which are uniform with respect to $\tau$:
\begin{align}
\label{bound.h}
\left\|h(\utau)\right\|_{L^\infty(0,T; L^1(\Omega))} &\leq C,\\
\label{bound.flux}
\left\| p(\Mtau)\nabla\sqrt{\frac{\utau_i q(\Mtau)}{p(\Mtau)}}\right\|_{L^2(0,T; L^2(\Omega))} &\leq C\quad (i=1,\ldots,n),\\
\label{bound.nau}
\left\|\sqrt{p(\Mtau)q(\Mtau)}\na\sqrt{\utau_i}\right\|_{L^2(0,T; L^2(\Omega))} &\leq C\quad (i=1,\ldots,n),\\
\label{bound.naM}
\left\|\na (\Mtau)^{\frac{a+1}{2}}\right\|_{L^2(0,T; L^2(\Omega))} &\leq C,\\
\label{bound.1minusM}
\left\|(1-\Mtau)^{1-b-\kappa}\right\|_{L^1(0,T; L^1(\Omega))} &\leq C,\\
\label{bound.wtau}
\tau^{1/2}\left\|\wtau\right\|_{L^2(0,T; H^m(\Omega))} &\leq C.
\end{align}
Moreover we recall that (by construction) $\utau(x,t)\in\cD$ a.e. $(x,t)\in\Omega\times (0,T)$, 
where $\cD$ is defined by \eqref{Def.D}. Therefore
\begin{equation}\label{bound.u.inf} 
\|\utau_i\|_{L^\infty(0,T; L^\infty(\Omega))}\leq C\quad (i=1,\ldots,n). 
\end{equation}
Furthermore, from \eqref{asy.q.Mto0}, \eqref{bound.nau} we get
\begin{equation}
\|(\Mtau)^{\frac{a-1}{2}}\na\utau_i\|_{L^2(Q_T)}\leq C \quad (i=1,\ldots,n).\label{bound.Mnau}
\end{equation}
We also point out that \eqref{esti.4}, \eqref{bound.h} imply
\begin{equation}
(1-\Mtau)^{1-\kappa}\in L^\infty(0,T; L^1(\Omega)).\label{bound.1minusM.2}
\end{equation}
The discretized-regularized system \eqref{d.r.eq} can be rewritten, in the new notation, as
\begin{align}\label{eq.utau}
& \sum_{i=1}^n\int_0^T \int_\Omega \left((D_\tau\utau_i)\phi_i + \alpha_i p^2(\Mtau)\na\left( \frac{\utau_i q(\Mtau)}{p(\Mtau)} \right)\cdot\na\phi_i - r_i(\utau)\phi_i\right)dxdt\\
&\qquad + \tau \sum_{i=1}^n\int_0^T(\wtau_i,\phi_i)_{H^m}dt = 0, \nonumber
\end{align}
for piecewise constant-in-time functions $\phi : [0,T]\to H^m_D(\Omega;\R^n)$. However, thanks to a standard density argument, \eqref{eq.utau} holds for all 
$\phi\in L^2(0,T; H^m_D(\Omega;\R^n))$. 

Next, we wish to find a $\tau-$uniform bound for $D_\tau\utau$. We first estimate the term $p^2(\Mtau)\na\left( \utau_i q(\Mtau)/p(\Mtau) \right)$.
We distinguish two cases.\\
{\em Case 1:} when $\kappa\geq b-1$, then $p(\Mtau)q(\Mtau)$ is bounded in $L^\infty(Q_T)$ thanks to \eqref{est.mto1}. It follows
\begin{align*}
& \left\|p^2(\Mtau)\na\left( \frac{\utau_i q(\Mtau)}{p(\Mtau)} \right)\right\|_{L^{2}(Q_T)}\\
&\leq 2\left\|\sqrt{\utau_i p(\Mtau)q(\Mtau)}\right\|_{L^{\infty}(Q_T)} \left\|p(\Mtau)\na\sqrt{\frac{\utau_i q(\Mtau)}{p(\Mtau)}}\right\|_{L^2(Q_T)}\\
&\leq 2\left\|p(\Mtau)q(\Mtau)\right\|_{L^{\infty}(Q_T)}^{1/2} \left\|p(\Mtau)\na\sqrt{\frac{\utau_i q(\Mtau)}{p(\Mtau)}}\right\|_{L^2(Q_T)} .
\end{align*}
{\em Case 2:} if $\kappa < b-1$, then $p(\Mtau)q(\Mtau)$ is bounded in $L^\frac{b+\kappa-1}{b-\kappa-1}(Q_T)$ due to \eqref{est.mto1} and \eqref{bound.1minusM}.
This leads to
\begin{align*}
& \left\|p^2(\Mtau)\na\left( \frac{\utau_i q(\Mtau)}{p(\Mtau)} \right)\right\|_{L^{1+\frac{\kappa}{b-1}}(Q_T)}\\
&\leq 2\left\|\sqrt{\utau_i p(\Mtau)q(\Mtau)}\right\|_{L^{\frac{2(b+\kappa-1)}{b-\kappa-1}}(Q_T)} \left\|p(\Mtau)\na\sqrt{\frac{\utau_i q(\Mtau)}{p(\Mtau)}}\right\|_{L^2(Q_T)}\\
&\leq 2\left\|p(\Mtau)q(\Mtau)\right\|_{L^{\frac{b+\kappa-1}{b-\kappa-1}}(Q_T)}^{1/2} \left\|p(\Mtau)\na\sqrt{\frac{\utau_i q(\Mtau)}{p(\Mtau)}}\right\|_{L^2(Q_T)} .
\end{align*}
The above estimates and \eqref{bound.flux} allow us to deduce
\begin{equation}
\left\|p^2(\Mtau)\na\left( \frac{\utau_i q(\Mtau)}{p(\Mtau)} \right)\right\|_{L^{1+\rho}(Q_T)}\leq C\quad (i=1,\ldots,n),\quad
\rho = \min\left\{1,\frac{\kappa}{b-1}\right\}.
\label{bound.flux2}
\end{equation}
As a byproduct of the above calculations we also get the following uniform bound:
\begin{equation}\label{bound.upq12}
\begin{cases}
\left\|\sqrt{\utau_i p(\Mtau)q(\Mtau)}\right\|_{L^{\infty}(Q_T)}\leq C, & \kappa\geq b-1,\\[1cm]
\left\|\sqrt{\utau_i p(\Mtau)q(\Mtau)}\right\|_{L^{\frac{2(b+\kappa-1)}{b-\kappa-1}}(Q_T)}\leq C, & \kappa < b-1.
\end{cases}
\end{equation}
Let us now deal with the reaction term. From \eqref{hp.r3}, \eqref{hp.r4} we deduce in particular that $|r_i(\utau)|\leq C (1-\Mtau)^{-\eta}$ with
$\eta<b+\kappa-1$. Therefore \eqref{bound.1minusM} leads to
\begin{align}
\|r_i(\utau)\|_{L^\frac{b+\kappa-1}{\eta}(Q_T)}\leq C.\label{bound.r}
\end{align}
From \eqref{bound.wtau}, \eqref{bound.flux2}, \eqref{bound.r} it follows
\begin{align}\label{est.Du}
&\int_0^T\int_\Omega (D_\tau\utau_i)\phi dx dt\\
\nonumber
&\leq C\left(\|\phi\|_{L^{\frac{1+\rho}{\rho}}(0,T; W^{1,\frac{1+\rho}{\rho}}(\Omega))} 
+ \|\phi\|_{L^\frac{b+\kappa-1}{b+\kappa-1-\eta}(Q_T)} + \tau^{1/2}\|\phi\|_{L^2(0,T; H^m(\Omega))}\right),
\end{align}
for $i=1,\ldots,n$. 
This estimate, together with
the Sobolev embedding $H^m(\Omega)\hookrightarrow L^\infty(\Omega)$ and the trivial relation $L^\infty(\Omega)\hookrightarrow L^\frac{b+\kappa-1}{b+\kappa-1-\eta}(\Omega)$,
implies that $D_\tau\utau_i$ (and also $D_\tau\Mtau$) is uniformly bounded in 
$L^{1+\eps}(0,T; (W^{1,\frac{1+\rho}{\rho}}(\Omega)\cap H^m(\Omega))')$, for $i=1,\ldots,n$
and some $\eps>0$. \medskip\\
{\bf Step 4: Limit $\tau\to 0$.} The uniform bound for $D_\tau\Mtau$ in $L^2(0,T; H^m(\Omega)')$, together with \eqref{bound.naM} and \eqref{bound.u.inf},
allows us to apply \cite[Theorem 3]{CJL14} with $Q(s)=s^{\frac{a+3}{2}}$ and deduce
\begin{align}\label{cnv.M}
\Mtau \to M \quad \mbox{strongly in $L^s(\Omega \times (0,T))$, for all }s<\infty.
\end{align}
The strong convergence of $\Mtau$ and bound \eqref{bound.1minusM} yield via Fatou's lemma that $(1-M)^{1-b-\kappa}\in L^1(Q_T)$, and therefore
$$ M<1\qquad\mbox{a.e. in }Q_T . $$
Moreover, thanks to the uniform bound for $D_\tau\utau_i$ in $L^{1+\eps}(0,T; (W^{1,\frac{1+\rho}{\rho}}(\Omega)\cap H^m(\Omega))')$ 
and estimates \eqref{bound.naM}, \eqref{bound.u.inf}, \eqref{bound.Mnau}, 
we can apply \cite[Lemma 7]{ZaJu17} to deduce that
\begin{align}\label{cnv.Mu}
(\Mtau)^{\frac{a+1}{2}}\utau_i\to M^{\frac{a+1}{2}}u_i \quad \mbox{strongly in $L^s(Q_T)$, for every $s<\infty$, $i=1,\ldots,n$.}
\end{align}
In particular, $(\Mtau)^{\frac{a+1}{2}}\utau_i\to M^{\frac{a+1}{2}}u_i$ a.e.~in $Q_T$, which implies
$$ \utau_i = \frac{(\Mtau)^{\frac{a+1}{2}}\utau_i}{(\Mtau)^{\frac{a+1}{2}}}\to\frac{M^{\frac{a+1}{2}}u_i}{M^{\frac{a+1}{2}}} = u_i
\quad\mbox{a.e.~in }Q_T\cap\{M>0\}. $$
Moreover, since $0\leq\utau_i\leq\Mtau$ for $1\leq i\leq n$, clearly
$$ \utau_i\to 0\quad\mbox{a.e.~in }Q_T\cap\{M=0\}\quad (1\leq i\leq n). $$
However, $u_i=0$ on $Q_T\cap\{M=0\}$. In fact, given any nonnegative $\phi\in L^2(Q_T)$ having support contained in $\{M=0\}$, it holds
$$ 0\leq \int_{Q_T}\utau_i \phi\, dx dt\leq \int_{Q_T}\Mtau\phi\, dx dt\to  \int_{Q_T}M\phi\, dx dt = 0 , $$
implying that the weak limit $u_i$ of $\utau_i$ vanishes on $Q_T\cap\{M=0\}$.
Summarizing up, by dominated convergence,
\begin{align}\label{cnv.u}
\utau_i\to u_i\quad\mbox{strongly in }L^s(Q_T)\quad\forall s<\infty,\quad 1\leq i\leq n.
\end{align}
From \eqref{bound.flux2} it follows that $p(\Mtau)^2\na(\utau_i q(\Mtau)/p(\Mtau))$ is weakly convergent in $L^{1+\rho}(Q_T)$,
where $\rho = \min\{1,\kappa/(b-1)\}$. However,
\begin{equation}\label{sat.id}
p^2(\Mtau)\na\left( \frac{\utau_i q(\Mtau)}{p(\Mtau)} \right)
= 2\sqrt{\utau_i p(\Mtau)q(\Mtau)}\, p(\Mtau)\na\sqrt{\frac{\utau_i q(\Mtau)}{p(\Mtau)}}. 
\end{equation}
Let us consider the first factor on the right-hand side of \eqref{sat.id}, i.e. $\sqrt{\utau_i p(\Mtau)q(\Mtau)}$.
The a.e.~convergence of $\utau_i$, $\Mtau$ and the fact that $M<1$ a.e.~in $Q_T$
imply that $\sqrt{\utau_i p(\Mtau)q(\Mtau)}\to \sqrt{u_i p(M)q(M)}$ a.e.~in $Q_T$.
Bound \eqref{bound.upq12} allows us to conclude that
\begin{equation}\label{cnv.upq}
\sqrt{\utau_i p(\Mtau)q(\Mtau)}\to \sqrt{u_i p(M)q(M)} \qquad\mbox{strongly in $L^2(Q_T)$.} 
\end{equation}
From \eqref{bound.flux} it follows
\begin{equation}\label{cnv.flux.a}
p(\Mtau)\na\sqrt{\frac{\utau_i q(\Mtau)}{p(\Mtau)}}\rightharpoonup \psi_i\quad\mbox{weakly in }L^2(Q_T),\quad i=1,\ldots,n,
\end{equation}
for some function $\psi\in L^2(Q_T;\R^n)$. From \eqref{bound.flux}, \eqref{cnv.upq}, \eqref{cnv.flux.a} we deduce
\begin{align}\label{cnv.flux1}
 p(\Mtau)^2\na\left(\frac{\utau_i q(\Mtau)}{p(\Mtau)}\right) \rightharpoonup 2\sqrt{u_i p(M)q(M)}\,\psi_i\quad\mbox{weakly in }L^{1}(Q_T).
\end{align}
We wish to identify the function $\psi_i$. For an arbitrary $\phi\in C^\infty_c(Q_T)$ let us consider
\begin{align}\label{idlim.1}
&-\int_{Q_T}\phi\, p(\Mtau)\na\sqrt{\frac{\utau_i q(\Mtau)}{p(\Mtau)}} dx dt = J_1^{(\tau)} + J_2^{(\tau)},\\
& J_1^{(\tau)}=\int_{Q_T}\sqrt{\utau_i p(\Mtau)q(\Mtau)}\,\na\phi \, dx dt, \nonumber\\
& J_2^{(\tau)}= \int_{Q_T}\phi p'(\Mtau)\sqrt{\frac{\utau_i q(\Mtau)}{p(\Mtau)}}\,\na\Mtau \, dx dt .\nonumber
\end{align}
From \eqref{cnv.upq} it follows immediately that 
\begin{equation}
J_1^{(\tau)}\to \int_{Q_T}\sqrt{u_i p(M)q(M)}\,\na\phi\, dx dt.\label{cnv.J1}
\end{equation}
Let us now consider $J_2^{(\tau)}$:
\begin{align*}
J_2^{(\tau)} &= \int_{Q_T}\phi p'(\Mtau)\sqrt{\frac{\utau_i q(\Mtau)}{p(\Mtau)}}\,\na\Mtau \, dx dt\\
&= \int_{Q_T}\phi \sqrt{\utau_i} g(\Mtau) \frac{(\Mtau)^{\frac{a-1}{2}}\na\Mtau}{(1-\Mtau)^{\frac{1+b+\kappa}{2}}}dx dt ,\\
g(M) &\equiv M^{-\frac{a-1}{2}}(1-M)^{\frac{1+b+\kappa}{2}}\frac{p'(M)}{p(M)}\sqrt{p(M)q(M)}\qquad 0<M<1.
\end{align*}
From \eqref{asy.q.Mto0} it follows that $g$ in continuous in $[0,1)$. From the a.e. convergence of $\Mtau$ and the fact that
$M<1$ a.e. in $Q_T$ it follows that $g(\Mtau)\to g(M)$ a.e. in $Q_T$. On the other hand,
by exploiting \eqref{Def.p}, \eqref{est.mto1} the function $g$ can be estimated for $M\to 1$ as follows
\begin{align*}
|g(M)| &\leq C \sqrt{q(M)p(M)} \frac{|p'(M)|}{p(M)}(1-M)^{\frac{1+b+\kappa}{2}}\\
&\leq C(1-M)^{\frac{1+\kappa-b}{2}}(1-M)^{-1-\kappa}(1-M)^{\frac{1+b+\kappa}{2}}\\
&\leq C\qquad\mbox{as }M\to 1.
\end{align*}
This means that $g$ is bounded in $[0,1]$. This fact, together with the a.e. convergence of $g(\Mtau)$, implies that
$g(\Mtau)\to g(M)$ strongly in $L^s(Q_T)$ for every $s<\infty$.

Let us now consider the term
$$ \frac{(\Mtau)^{\frac{a-1}{2}}\na\Mtau}{(1-\Mtau)^{\frac{1+b+\kappa}{2}}} = \na\Phi(\Mtau),\qquad
\Phi(M)\equiv\int_0^M\frac{s^{\frac{a-1}{2}}}{(1-s)^{\frac{1+b+\kappa}{2}}}ds . $$
Lemma \ref{bound.nabla} and the entropy inequality \eqref{d.e.i} ensure that $\na\Phi(\Mtau)$ is uniformly bounded in $L^2(Q_T)$.
Moreover we know that $\Phi(\Mtau)\to\Phi(M)$ a.e.~in $Q_T$. Furthermore,
$$ \Phi(M)\leq\int_0^M\frac{ds}{(1-s)^{\frac{1+b+\kappa}{2}}} = \frac{(1-M)^{\frac{1-b-\kappa}{2}} - 1}{(b+\kappa-1)/2} , $$
which, thanks to \eqref{bound.1minusM}, implies that $\Phi(\Mtau)$ is uniformly bounded in $L^{2}(Q_T)$. This means that 
$\Phi(\Mtau)\to\Phi(M)$ strongly in $L^{2-\delta}(Q_T)$ for every $\delta>0$. As a consequence, we get that
$\na\Phi(\Mtau)\rightharpoonup\na\Phi(M)$ weakly in $L^2(Q_T)$. 
The weak convergence of $\na\Phi(\Mtau)$, the strong convergence of $g(\Mtau)$, and the strong convergence of $\utau$ \eqref{cnv.u} allow us to conclude
\begin{equation}
J_2^{(\tau)}\to \int_{Q_T}\phi p'(M)\sqrt{\frac{u_i q(M)}{p(M)}}\,\na M \, dx dt. \label{cnv.J2}
\end{equation}
By putting \eqref{cnv.flux.a}--\eqref{cnv.J2} together we conclude
\begin{align}\label{cnv.flux12}
p(\Mtau)\na\sqrt{\frac{\utau_i q(\Mtau)}{p(\Mtau)}} &\rightharpoonup p(M)\na\sqrt{\frac{u_i q(M)}{p(M)}}
 \quad\mbox{weakly in }L^2(Q_T),\\
 \label{cnv.flux}
 p(\Mtau)^2\na\left(\frac{\utau_i q(\Mtau)}{p(\Mtau)}\right) &\rightharpoonup 
 p(M)^2\na\left(\frac{u_i q(M)}{p(M)}\right) \quad\mbox{weakly in }L^1(Q_T),
\end{align}
for $i=1,\ldots,n$.
Estimate \eqref{est.Du} yields 
\begin{equation}\label{cnv.Du}
D_\tau\utau_i\rightharpoonup\pa_t u_i \quad\mbox{weakly in }L^{1+\eps}(0,T; (W^{1,\frac{1+\rho}{\rho}}(\Omega)\cap H^m(\Omega))'),~~i=1,\ldots,n.
\end{equation}
Let us now study the convergence of the reaction term. From \eqref{bound.r} and the fact that
$(b+\kappa-1)/\eta>1$ (by assumption \eqref{hp.r4}),
if we can prove the a.e.~convergence of $r_i(\utau)$ in $Q_T$ then strong convergence in a suitable space will follow.
However, we know that $\utau\to u$ a.e.~in $Q_T$ and $r$ is continuous in $\{M<1\}$;
therefore $r(\utau)\to r(u)$ a.e.~in $Q_T$.
We conclude that $r(\utau)\to r(u)$ strongly in $L^1(Q_T)$.

Finally, $\tau\wtau\to 0$ strongly in $L^2(0,T; H^m(\Omega))$ thanks to \eqref{bound.wtau}. We conclude that 
we can take the limit $\tau\to 0$ in \eqref{eq.utau} and obtain
\begin{align}\label{eq.final}
& \sum_{i=1}^n\langle\pa_t u_i,\phi_i\rangle
+\sum_{i=1}^n\int_\Omega \left(\alpha_i p^2(M)\na\left( \frac{u_i q(M)}{p(M)} \right)\cdot\na\phi_i - r_i(u)\phi_i\right)dx=0,
\end{align}
for every $\phi = (\phi_1,\ldots,\phi_n)\in C^\infty_c(Q_T;\R^n)$.
However, \eqref{bound.flux2} and \eqref{bound.r} allow us to deduce via a density argument that
\eqref{eq.final} holds for all $\phi\in L^\frac{\rho+1}{\rho}(0,T; W^{1,\frac{\rho+1}{\rho}}(\Omega))\cap L^{\frac{b+\kappa-1}{b+\kappa-1-\eta}}(Q_T)$.\medskip\\
{\bf Step 5: Entropy inequality.}
Testing \eqref{d.e.i} against an arbitrary test function $\phi\in C^1_c(0,T)$ (and performing a ``discrete integration by parts'') leads to
 \begin{align}\label{d.e.i.int}
\int_0^T \int_\Omega h^*(\utau|u_D) D_{-\tau}\phi\, dx dt 
&+ 2\sum_{i=1}^n\alpha_i\int_0^T\int_\Omega p(\Mtau)^2\left|\nabla\sqrt{\frac{\utau_i q(\Mtau)}{p(\Mtau)}}\right|^2\phi dx dt \\
& \leq \int_0^T\left(C_1\int_\Omega h^*(\utau|u_D) dx  + C_2\right)\phi\ dt .\nonumber
\end{align}
From \eqref{cnv.flux12} and the weakly lower semicontinuity of the $L^2$ norm it follows
\begin{align*}
&\liminf_{\tau\to 0}\sum_{i=1}^n\int_0^T\int_\Omega p(\Mtau)^2\left|\nabla\sqrt{\frac{\utau_i q(\Mtau)}{p(\Mtau)}}\right|^2\phi dx dt\\
&\qquad \geq\sum_{i=1}^n\int_0^T\int_\Omega p(M)^2\left|\nabla\sqrt{\frac{u_i q(M)}{p(M)}}\right|^2\phi dx dt 
\end{align*}
On the other hand, \eqref{esti.4} and the definition of $h$ implies (via l'H\^{o}pital's rule) that
$h(u)\leq C(1-M)^{1-\kappa}$. Since \eqref{bound.1minusM} holds, we deduce that $h(\utau)$ is bounded in
$L^{1+\delta}(Q_T)$ for some $\delta>0$. The a.e.~convergence of $\utau$, the fact that $M<1$ a.e. in $Q_T$, and the continuity of $h$
in $\{M<1\}$ implies the a.e. convergence of $h(\utau)$ towards $h(u)$. We conclude that
\begin{equation}
h^*(\utau\vert u_D)\to h^*(u\vert u_D)\quad\mbox{strongly in }L^{1+\delta/2}(Q_T).\label{cnv.h}
\end{equation}
Since $D_{-\tau}\phi = -\tau^{-1}(\phi(\cdot + \tau)-\phi)\to -\pa_t\phi$ strongly in $L^s(0,T)$ for all $s<\infty$,
from \eqref{d.e.i.int}--\eqref{cnv.h} we get
\begin{align}\label{d.e.i.int.2}
&-\int_0^T \int_\Omega h^*(u|u_D) \pa_t\phi\, dx dt 
+ 2\alpha_i\sum_{i=1}^n\int_0^T\int_\Omega p(M)^2\left|\nabla\sqrt{\frac{u_i q(M)}{p(M)}}\right|^2\phi dx dt \\
&\qquad \leq \int_0^T\left(C_1\int_\Omega h^*(u|u_D) dx  + C_2\right)\phi\ dt ,
\quad\forall\phi\in C^1_c(0,T): ~\phi\geq 0.\nonumber
\end{align}
Therefore \eqref{entr.inequ.2} holds. This finishes the proof of the existence Theorem.
$\hfill \Box$
\begin{remark}\label{Rem.Dir}
\textnormal{Theorem \ref{Thm.Existence} can be proved also in the case of 
nonconstant, $x-$dependent Dirichlet boundary data $u_{D,i}=u_{D,i}(x)$ ($i=1,\ldots,n$).
We can assume for the sake of simplicity that $u_{D,i}\in W^{1,\infty}(\Omega)$, $i=1,\ldots,n$, and $\sup_{\Omega}M_D < 1$.
The only relevant difference with the case $u_D=$constant lies in the proof of Lemma \ref{lem.dei} (i.e.~the discrete entropy inequality).
In fact, the following additional term appears on the right-hand side of \eqref{d.e.i.1}:
\begin{align*}
\Xi = \sum_{i=1}^n \alpha_i\int_0^T\int_\Omega\na\log\left( \frac{u_{D,i}q(M_D)}{p(M_D)} \right)\cdot p^2(\Mtau)\na\left( \frac{\utau_i q(\Mtau)}{p(\Mtau)} \right)dxdt .
\end{align*}
This term can be estimated as follows:
\begin{align*}
\Xi &= 2\sum_{i=1}^n \alpha_i\int_0^T\int_\Omega\na\log\left( \frac{u_{D,i}q(M_D)}{p(M_D)} \right)\cdot \sqrt{\utau_i p(\Mtau)q(\Mtau)}\, p(\Mtau)\na\sqrt{\frac{\utau_i q(\Mtau)}{p(\Mtau)} }dxdt\\
&\leq 2\sum_{i=1}^n \alpha_i\left\| \na\log\left( \frac{u_{D,i}q(M_D)}{p(M_D)} \right) \right\|_{L^\infty(Q_T)}
\left\| \sqrt{\utau_i p(\Mtau)q(\Mtau)} \right\|_{L^2(Q_T)}\\
&\qquad \times \left\|p(\Mtau)\na\sqrt{\frac{\utau_i q(\Mtau)}{p(\Mtau)} } \right\|_{L^2(Q_T)}\\
&\leq \frac{C}{\eps}\left\| \sqrt{\utau_i p(\Mtau)q(\Mtau)} \right\|_{L^2(Q_T)}^2 
+ C\eps \left\|p(\Mtau)\na\sqrt{\frac{\utau_i q(\Mtau)}{p(\Mtau)} } \right\|_{L^2(Q_T)}^2 ,
\end{align*}
for some $\eps>0$. However, thanks to \eqref{est.mto1}, it holds
\begin{align*}
\left\| \sqrt{\utau_i p(\Mtau)q(\Mtau)} \right\|_{L^2(Q_T)}^2 &\leq C \left\| (1-\Mtau)^{(1+\kappa-b)/2} \right\|_{L^2(Q_T)}^2\\
&= C \int_0^T\int_\Omega (1-\Mtau)^{1+\kappa-b} dx dt  .
\end{align*}
We wish to estimate the last integral in the above inequality. It holds
\begin{align*}
 &\int_0^T\int_\Omega (1-\Mtau)^{1+\kappa-b} dx dt\\ 
 & \qquad  =\iint_{\{1-\Mtau\geq\eps^{\frac{1}{\kappa}}\}} (1-\Mtau)^{1+\kappa-b} dx dt
 +\iint_{\{1-\Mtau < \eps^{\frac{1}{\kappa}}\}} (1-\Mtau)^{1+\kappa-b} dx dt\\
 &\qquad\leq C(\eps) + \eps^2\iint_{\{1-\Mtau < \eps^{\frac{1}{\kappa}}\}} (1-\Mtau)^{1-\kappa-b} dx dt\\
 &\qquad\leq C(\eps) + \eps^2\int_0^T\int_\Omega (1-\Mtau)^{1-\kappa-b} dx dt .
\end{align*}
Thus, we deduce that
\begin{align*}
\Xi & \leq C(\eps) + C\eps\int_0^T\int_\Omega (1-\Mtau)^{1-\kappa-b} dx dt + 
C\eps \left\|p(\Mtau)\na\sqrt{\frac{\utau_i q(\Mtau)}{p(\Mtau)} } \right\|_{L^2(Q_T)}^2.
\end{align*}
Lemma \ref{bound.nabla} allows us to bound $\Xi$ by means of the entropy dissipation by choosing $\eps>0$ small enough,
thereby yielding Lemma \ref{lem.dei}. The rest of the proof is analogue to the case $u_D = $constant.}
\end{remark}

\begin{remark}\label{remark.neu.bc}
\textnormal{The existence proof works also in the case \eqref{cond.b.0} is replaced by homogeneous Neumann boundary conditions on the whole $\pa\Omega$. 
In this case the relative entropy \eqref{H} cannot be employed (as $u_D$ is
obviously not defined), and the absolute entropy $H[u]=\int_\Omega h(u)dx$ can be used instead. However, the assumptions on the reaction term need to be modified,
as the total mass 
\begin{align}\label{total.mass}
\mathcal{M}(t) = |\Omega|^{-1}\int_\Omega M(x,t)dx 
\end{align}
needs to remain strictly smaller then 1 for any finite time
in order to prevent blowup. A sufficient condition for this reads as
\begin{align}\label{hp.r.Neumann}
\exists C\in\R : \qquad \sum_{i=1}^n r_i(u)\leq C(1-M)\qquad \forall u\in\cD.
\end{align}
By integrating \eqref{cross.diff.sys.00} in $\Omega$, summing from $i=1,\ldots,n$, and exploiting \eqref{hp.r.Neumann}, one can easily show that
$$1-\mathcal{M}(t)\geq e^{-\lambda t}(1-\mathcal{M}(0)) \qquad t>0, \qquad\lambda \equiv \max\{C,0\}, $$
where $\mathcal{M}(t)$ is the total mass defined in \eqref{total.mass}.
This control on the total mass $\mathcal{M}(t)$ allows us to apply Lemma \ref{lem.Poi.sing} with $M=\Mtau$ and obtain
bound \eqref{bound.1minusM} for the singularity, which is the only delicate point in the existence proof for the case of homogeneous Neumann boundary conditions; 
the rest of the argument works in a completely analogue way to the case of mixed Dirichlet-Neumann boundary conditions.}
\end{remark}

\section{Proof of Theorem~\ref{Thm.LTB}}\label{Sec.LTB}

From the definition \eqref{H}, \eqref{Entropy} of the relative entropy density it follows
\begin{align*}
   h^*(u|u_D)  
  & = \sum_{i=1}^n \big( u_i  \log \frac{u_i}{u_{D,i}} + u_{D,i} - u_i \big) 
   +  \int_{M_D}^{M} \log \Big(  \frac{q(s)/p(s)}{q(M_D)/p(M_D)} \Big) ds.
\end{align*}

Now we split the above written relative entropy density 
 in  two parts,
\[ h^*(u|u_D) = \sum_{i=1}^n h_1^*(u_i|u_{D,i}) + h_2^*(M|M_D), \]  where
\begin{align*}
h_1^*(u|u_D) & = u \log \frac{u}{u_{D}}  + u_{D} - u,\\
h_2^*(M|M_D) & = \int_{M_D}^{M} \log \Big(  \frac{q(s)/p(s)}{q(M_D)/p(M_D)} \Big) ds.\\
\end{align*}
The entropy inequality \eqref{entr.inequ.2} and Lemma \ref{bound.nabla} yield, since $\lambda_r=C_r=0$, 
 \begin{align}
  \frac{d}{dt} \int_\Omega h^*(u|u_D) dx  + C_1 \sum_{i=1}^n \int_\Omega p(M) q(M) & |\nabla \sqrt{u_i}|^2 dx
   \nonumber \\
   & + C_2 \int_\Omega \frac{M^{a-1}}{(1-M)^{1+b+\kappa}}|\nabla M|^2 dx \leq 0.
 \label{LTB: RelEntropy}
 \end{align}
We want to  estimate the integral
\begin{align*}
 \int_\Omega p(M) q(M) |\nabla \sqrt{u_i}|^2 dx 
\end{align*}
from below using the term $\int_\Omega h_1^*(u_i|u_{D,i}) dx$.
We start by observing that, thanks to \eqref{est.mto1} and \eqref{asy.q.Mto0}, it holds
\[ p(M) q(M) \geq C M^a (1-M)^{1+\kappa-b}. \]
It follows
\begin{align}
  \int_\Omega p(M) q(M) |\nabla \sqrt{u_i}|^2 dx 
    & \geq C \int_\Omega M^a (1-M)^{1+\kappa - b} |\nabla \sqrt{u_i}|^2 dx \nonumber \\
    & \geq C \int_\Omega u_i^a (1-M)^{1+\kappa - b} |\nabla \sqrt{u_i}|^2 dx \nonumber \\
    & \geq C \int_\Omega u_i^{a-1} (1-M)^{1+\kappa - b} |\nabla u_i|^2 dx \nonumber \\
    & = C \int_\Omega  (1-M)^{1+\kappa - b} | \nabla  \big( u_i^{\frac{a+1}{2}} \big)|^2 dx.
\end{align}
On the other side, we notice that the term
 $\int_\Omega |\nabla \big( u_i^{\frac{a+1}{2}} \big) |dx   $
can be estimated through the Cauchy-Schwarz inequality as follows
\begin{align}
 \int_\Omega |\nabla \big( u_i^{\frac{a+1}{2}} \big) |dx  & = \int_\Omega 
  \frac{(1-M)^{(1+\kappa-b)/2}|\nabla u_i^{(a+1)/2}|}{(1-M)^{(1+\kappa-b)/2}} dx \nonumber \\
 & \leq \Big( \int_\Omega (1-M)^{-1-\kappa + b} dx \Big)^{1/2} \Big( \int_\Omega (1-M)^{1+\kappa-b} |\nabla u_i^{(a+1)/2}|^2 dx   \Big)^{1/2}.
 \label{LTB:Eq.2}
\end{align}
The first integral in the last row of \eqref{LTB:Eq.2} can be controlled by means of the $L^\infty(L^1)$ bound on $(1-M)^{1-\kappa}$:
\begin{align*}
\int_\Omega (1-M)^{-1-\kappa+b} dx \leq \int_\Omega (1-M)^{1-\kappa} dx \leq C,
\end{align*}
where we used the fact that $-1-\kappa+b \geq 1-\kappa$ which holds true for $b\geq 2$.
In this way we get
\begin{align}\label{LTB:Eq.2b}
  \int_\Omega p(M) q(M) |\nabla \sqrt{u_i}|^2 dx  \geq C \Big(  \int_\Omega |\nabla u_i^{(a+1)/2} | dx \Big)^2 .
\end{align}
Finally, by using the Poincar\'e inequality we get
\begin{align}
   \int_\Omega p(M) q(M) |\nabla \sqrt{u_i}|^2 dx  \geq C \Big( \int_\Omega |u_i^{(a+1)/2} - (u_{D,i})^{(a+1)/2}| \,dx   \Big)^2.
  \label{LTB:Eq.3}
\end{align}
Next step is to bound term $|u_i^{(a+1)/2} - (u_{D,i})^{(a+1)/2}|$ from below by $h_1^*(u_i|u_{D,i})$.
For $u_i\leq u_{D,i}/2$ it holds
\begin{align*}
|u_i^{(a+1)/2} - (u_{D,i})^{(a+1)/2}| = (u_{D,i})^{(a+1)/2} - u_i^{(a+1)/2} \geq C\geq C |u_i - u_{D,i}| .
\end{align*}
Moreover,
\begin{align*}
h_1^*(u_i|u_{D,i}) & = u_i \log \frac{u_i}{u_{D,i}}  + u_{D,i} - u_i\leq u_{D,i} - u_i\leq |u_i - u_{D,i}| .
\end{align*}
On the other hand, for $u_i\geq u_{D,i}/2$, 
by the mean-value theorem and the Taylor's formula,
there exist $\xi_i^{(1)}, \xi_i^{(2)}$ intermediate between $u_i$ and $u_{D,i}$ such that, 
\begin{align*}
|u_i^{(a+1)/2} - (u_{D,i})^{(a+1)/2}| &= \frac{a+1}{2}\left(\xi_i^{(1)}\right)^{(a-1)/2}|u_i - u_{D,i}|\geq C |u_i - u_{D,i}| ,\\
h_1^*(u_i|u_{D,i}) & = \frac{1}{2\xi_i^{(2)}}(u_i - u_{D,i})^2\leq C |u_i - u_{D,i}|.
\end{align*}
So in both cases
\begin{equation}
|u_i^{(a+1)/2} - (u_{D,i})^{(a+1)/2}|\geq C|u_i - u_{D,i}|\geq C h_1^*(u_i|u_{D,i})
\quad (i=1,\ldots,n).\label{est.ua.h}
\end{equation}
From \eqref{LTB:Eq.3}, \eqref{est.ua.h} we deduce
\begin{align}
 \int_\Omega p(M) q(M) |\nabla \sqrt{u_i}|^2 dx \geq C \Big( \int_\Omega h_1^*(u|u_D) dx \Big)^2
\label{LTB:Eq.5}
\end{align}
Next step is to estimate the third integral in \eqref{LTB: RelEntropy} from below by the integral 
\[ \int_\Omega h_2^*(M|M_D) dx. \] 
For that purpose we define
\begin{align}
 \Phi(M) := \int_{M_D}^M \frac{s^{(a-1)/2}}{(1-s)^{(1+b+\kappa)/2}} ds,\quad 0 \leq M < 1.
\label{LTB:def.Phi}
\end{align}
Now, one gets 
\begin{align}
 \int_\Omega \frac{M^{a-1}}{(1-M)^{1+b+\kappa}} |\nabla M|^2 dx = \int_\Omega |\nabla \Phi(M)|^2 dx \geq C \int_\Omega |\Phi(M)|^2 dx,
\label{LTB:Eq.6}
\end{align}
where we used the Poincar\'e inequality in order to make the last estimate.

Further, we claim that
\begin{align}
  h_2^*(M|M_D) \leq C \Phi(M)^2
  \label{LTB:Eq.7.0}
\end{align}
i.e.
\begin{align}
 \int_{M_D}^M \log\Big(   \frac{q(s)/p(s)}{q(M_D)/p(M_D)}\Big) ds \leq C \Big( \int_{M_D}^M \frac{s^{(a-1)/2}}{(1-s)^{(1+b+\kappa)/2}}\,ds\Big)^2.
\label{LTB:Eq.7}
\end{align}
For checking the claim \eqref{LTB:Eq.7.0} we need to show that the function
\begin{align}
 F(M) := \frac{h_2^*(M|M_D) }{\Phi(M)^2}
\label{LTB:Eq.7.1}
\end{align}
is bounded.
Note that for $0 \leq M < 1$ and $M \neq M_D$ it is clear that $F \in C\big( [0,1) \setminus \{M_D\} \big)$.
It remains to check the behavior of function $F$ near $M=1$ and $M=M_D$.
By using the estimate $s^{(a-1)/2} \geq M_D^{(a-1)/2}$
inside the integral defining $\Phi$ and noticing that $\Phi'(M_D)\neq 0$ we deduce
$$ \Phi(M)\sim |M-M_D|\quad (M\to M_D),\qquad \Phi(M)\sim (1-M)^{\frac{1-b-\kappa}{2}}\quad (M\to 1). $$
These relations, estimate \eqref{esti.4} and the fact that $h_2^*(M_D|M_D)=\frac{\pa}{\pa M}h_2^*(M|M_D)\vert_{M=M_D}=0$ allow us to obtain
\begin{align*}
 F(M) &\sim \frac{(1-M)^{1-\kappa}}{(1-M)^{1-\kappa-b}} \sim (1-M)^b \to 0\quad (M\to 1),
\end{align*}
and that $F$ is bounded as $M\to M_D$.
It follows that function $F$ is bounded which proves our claim \eqref{LTB:Eq.7.0}.

In this way we get
\begin{align}
 \int_\Omega \frac{M^{a-1}}{(1-M)^{1+b+\kappa}} |\nabla M|^2 dx \geq C \int_\Omega h_2^*(M|M_D) dx.
\label{LTB:Eq.8}
\end{align}
On the other side, one has
\begin{align}
 \int_\Omega  h_2^*(M|M_D) dx \leq \int_\Omega h^*(u|u_D) dx \leq \int_\Omega h^*(u_0|u_D) dx = C > 0,
 \label{LTB:Eq.9}
\end{align}
so \eqref{LTB:Eq.8} implies a fortiori that
\begin{align}
 \int_\Omega \frac{M^{a-1}}{(1-M)^{1+b+\kappa}} |\nabla M|^2 dx 
  \geq C \Big( \int_\Omega h_2^*(M|M_D)  dx \Big)^2.
\label{LTB:Eq.10}
\end{align}
 Finally, we collect estimates \eqref{LTB:Eq.5} and \eqref{LTB:Eq.10} and we combine them with the entropy inequality \eqref{LTB: RelEntropy}. We get
 \begin{align}
 \frac{d}{dt}\int_\Omega h^*(u|u_D) dx + C \Big( \int_\Omega h^*(u|u_D) dx \big)^2 \leq 0.
 \label{LTB:Eq.11}
 \end{align}
 Let us denote
 \[ w(t) = \int_\Omega h^*(u|u_D) dx. \]
 Now, equation \eqref{LTB:Eq.11} can be written as
 \begin{align}
  \frac{d}{dt} w(t) + C w(t)^2 \leq 0.
  \label{LTB:Eq.12}
 \end{align}
 By integrating \eqref{LTB:Eq.12} with respect to time, from $0$ to $t$ ($t>0$), one gets
 \[ \frac{1}{H_0} - \frac{1}{w(t)} \leq -Ct, \]
 where $H_0 = \int_\Omega h^*
 (u_0|u_D) dx$. 
 Now, direct calculations give
 \begin{align}
  \int_\Omega h^*(u|u_D) dx  \leq \frac{H_0}{1+t C H_0}.
 \label{LTB:Eq.13}
 \end{align}
 Since $h^*(u|u_D)\geq \sum_{i=1}^n h^*_1(u_i|u_{D,i})$ and the Hessian of $u\mapsto h^*_1(u|u_D)$ is uniformly positive definite for $0<u\leq 1$, by Taylor-expanding $h_1^*(u_i|u_{D,i})$ around $u_{D,i}$ we conclude
\begin{align*}
\int_\Omega h^*(u|u_D) dx\geq \sum_{i=1}^n\int_\Omega h^*_1(u_i|u_{D,i}) dx\geq C \sum_{i=1}^n\int_\Omega |u_i-u_{D,i}|^2 dx .
\end{align*}
This finishes the proof of Theorem~\ref{Thm.LTB}. $\hfill \Box$
  
\begin{remark}\label{Remark.Neumann.LTB}
\textnormal{A similar (albeit weaker) result on the long-time behavior of the solutions to \eqref{cross.diff.sys.1} holds in the case of homogeneous Neumann boundary conditions
with vanishing reactions.
Indeed, one can show that, if $a\leq 1$, $\kappa<1$, $b\geq 2$, and $r(u)\equiv 0$, then $C>0$, $\theta\in (0,1)$ exist such that
$$ \|u-\langle u\rangle\|_{L^2(\Omega)}^2\leq \frac{C}{(1+t)^\theta}\qquad t>0, $$
where $\langle u\rangle = |\Omega|^{-1}\int_\Omega u dx = |\Omega|^{-1}\int_\Omega u_0 dx$ is the steady state of the system.
The assumption $r=0$ is made for the sake of simplicity; as a matter of fact, the result could be generalized to the case of nonzero reaction terms
with zero space average and suitable dissipative properties (like e.g. $r(u)\cdot \frac{\pa}{\pa u}h(u|\langle u\rangle)\leq 0$). However, such a case
seems quite artificial.}

\textnormal{The main differences with the case of mixed Dirichlet-Neumann boundary conditions appear in two points. 
The first one is the proof that the right-hand side of \eqref{LTB:Eq.2b} dominates $\int_\Omega h_1^*(u|\langle u \rangle)dx$.
In the case of homogeneous Neumann boundary conditions, Poincar\'e inequality yields (thanks to the assumption $a\leq 1$):
\footnote{We point out that, if $a>1$, the term $\langle u_i^{(a+1)/2}\rangle$ would be present in the inequality in place of $\langle u_i\rangle$; it is not clear how 
the argument would work in this case.}
$$ \int_\Omega |\na u_i^{(a+1)/2}|dx \geq c\int_\Omega |\na u_i|dx\geq c'\int_\Omega |u_i - \langle u_i\rangle|dx . $$
The above inequality, together with \eqref{LTB:Eq.2b}, \eqref{est.ua.h}, implies \eqref{LTB:Eq.5} with $u_D$ replaced by $\langle u\rangle$.}

\textnormal{The other delicate point in the proof is to relate the second integral on the right-hand side of \eqref{LTB: RelEntropy} with $\int_\Omega h_2^*(M|\langle M \rangle)dx$.
We start by noticing that the relation $a\leq 1$ and Poincar\'e inequality yield
\begin{equation}\label{LBThomNeu.1}
\int_\Omega \frac{M^{a-1}}{(1-M)^{1+b+\kappa}}|\nabla M|^2 dx\geq \int_\Omega |\nabla M|^2 dx \geq 
C\int_\Omega |M - \langle M\rangle|^2 dx. 
\end{equation}
Furthermore, since $\kappa<1$, then (thanks to \eqref{esti.4}) $s\mapsto\log(q(s)/p(s))$ is integrable in $[0,1]$. 
A straightforward consequence of this fact is the property
$$ h_2^*(M|\langle M\rangle)\leq C|M-\langle M\rangle|^{1-\kappa}, $$
for some constant $C>0$. By Jensen's inequality it follows
\begin{equation}\label{LBThomNeu.2}
\left(\int_\Omega h_2^*(M|\langle M\rangle)dx\right)^{\frac{2}{1-\kappa}}
\leq C\left(\int_\Omega |M-\langle M\rangle|^{1-\kappa} dx\right)^{\frac{2}{1-\kappa}}
\leq C\|M-\langle M\rangle\|_{L^2(\Omega)}^{2} . 
\end{equation}
Putting \eqref{LBThomNeu.1}, \eqref{LBThomNeu.2} together yields
$$ \int_\Omega \frac{M^{a-1}}{(1-M)^{1+b+\kappa}}|\nabla M|^2 dx\geq C\left( \int_\Omega h_2^*(M|\langle M\rangle)dx \right)^{\frac{2}{1-\kappa}}. $$
The rest of the proof is completely analogue to the case of mixed Dirichlet-Neumann boundary conditions.}
\end{remark}

  \section{Proof of Theorem~\ref{Thm.Uniqueness}}\label{Sec.Uniq}
 
  The uniqueness proof is organized in two parts. First, using the $H^{-1}$-method, we prove the uniqueness of the total mass $M$. Consequently, in order to show  the uniqueness of the solution $u=(u_1,\ldots,u_n))$ we apply the $E$-monotonicity technique of Gajewski \cite{Gaj94}.
By summing equations \eqref{cross.diff.sys.1}, 
taking into account the assumption  that $\alpha_i = 1$, and denoting $\di f(M) =  M q(M)/p(M)$, one gets
\begin{align}
\partial_t M = \Div\!\left(  p^2(M) \nabla f(M) \right) + \sum_{i=1}^n r_i(u).
\label{eq.sum}
\end{align}
Direct calculation gives
\begin{align*}
 p^2(M) \nabla f(M) = 
 \Big( p(M)q(M) + M p(M)q'(M) - Mq(M)p'(M) \Big) \nabla M.
\end{align*}
Let us define the function $Q$ as follows:
\begin{align}
 Q(M) = \int_0^M \big( p(s)q(s) + p(s)q'(s) s - p'(s) q(s) s \big)ds.
\label{def.Q}
\end{align}
In this way, equation \eqref{eq.sum} can be written as
\begin{align}
 \partial_t M = \Delta Q(M) + \sum_{i=1}^n r_i(u).
\label{eq.sum.1}
\end{align}
Let us assume that equation \eqref{eq.sum.1} has two different solutions $M_1$ and $M_2$.
By substracting equations, one gets
\begin{align}
 \partial_t (M_1-M_2) = \Delta\big( Q(M_1) - Q(M_2) \big) + \sum_{i=1}^n \big( r_i(u_1) - r_i(u_2) \big).
 \label{eq.sum.2}
\end{align}
We test equation \eqref{eq.sum.2} using the test-function $\varphi(t)$ which we choose in a special way.
More precisely, let $\varphi(t)  \in H^1(\Omega)$ be the unique solution to the following initial-boundary value problem:
\begin{equation}
\begin{split}
 -\Delta \varphi(t) & = (M_1 - M_2)(t) \; \textrm{ in }\; \Omega,\\
 \varphi & = 0\; \textrm{ on }\; \Gamma_D,\\
 \frac{\partial \varphi}{\partial \nu} & = 0 \; \textrm{ on }\; \Gamma_N,\\
 \varphi(0) & = \varphi_0 \; \textrm{ in }\; \Omega.
\label{test.1}
\end{split}
\end{equation}
By exploiting \eqref{eq.sum.2} and the fact that $\na\varphi(0)=0$ we deduce
\begin{align*}
 &\frac{1}{2}\|\na\varphi(t)\|_{L^2(\Omega)}^2 =  \int_0^t\frac{1}{2} \frac{d}{dt} \| \nabla \varphi(s)\|_{L^2(\Omega)}^2 ds 
  = -\int_0^t \langle \partial_t \Delta \varphi, \varphi \rangle ds 
  = \int_0^t\langle  \partial_t(M_1 - M_2) , \varphi \rangle ds\\
  &= \int_0^t\langle \Delta\big( Q(M_1) - Q(M_2) \big), \varphi \rangle ds + \int_0^t\langle \sum_{i=1}^n (r_i(u_1) - r_i(u_2), \varphi \rangle ds .
\end{align*}
By integrating by parts and applying \eqref{test.1} we get
\begin{align}
 \frac{1}{2} \|\nabla \varphi(t)\|_{L^2(\Omega)}^2 
 & = - \int_0^t \int_\Omega \big(  Q(M_1) - Q(M_2) \big)  (M_1 - M_2) dxds \nonumber \\
 & \qquad+ \int_0^t \int_\Omega \sum_{i=1}^n \big( r_i(u_1) - r_i(u_2) \big) \varphi dxds.
  \label{Uniq:Eq.4.1}
  \end{align}
  Let us first consider the second integral on the right-hand side of \eqref{Uniq:Eq.4.1}.
Using the assumption \eqref{assum.react.1} we have 
\begin{align}\label{Uniq:Eq.4.2}
&\int_0^t \int_\Omega \sum_{i=1}^n \big( r_i(u_1) - r_i(u_2) \big) \varphi dxds =  C J_1 + J_2 ,  \\
& J_1 = \int_0^t \int_\Omega (M_1 - M_2) \varphi dx ds ,\quad 
 J_2 = \int_0^t \int_\Omega \big( R(M_1) - R(M_2)\big) \varphi dx ds .\nonumber
\end{align}
  We calculate:
  \begin{align*}
  J_1 = -  \int_0^t \int_\Omega (\Delta \varphi)\, \varphi dx ds =  \int_0^t \int_\Omega |\nabla \varphi|^2 dx ds,
  \end{align*}
  and using the mean-value theorem and assumption \eqref{assum.react.2} we get
  \begin{align*}
  J_2  \leq  \int_0^t \int_\Omega |R'(\overline{M})| |M_1 - M_2| |\varphi| dx ds  
   \leq C  \int_0^t \int_\Omega \overline{M}^{a/2} |M_1 - M_2| |\varphi| dx ds,
  \end{align*}
  where $\overline{M} = \Theta M_1 + (1-\Theta)M_2$, for some $\Theta \in [0,1]$.
  Next, Young inequality gives
   \begin{align}
  J_2 \leq C(\eps)  \int_0^t \int_\Omega \varphi^2(s) dx ds + \eps  \int_0^t \int_\Omega  \overline{M}^a (M_1 - M_2)^2 dx ds.
  \label{Uniq:Eq.4.3}
  \end{align}
Further on, using the Poincar\'e inequality in the first integral of \eqref{Uniq:Eq.4.3} and the estimate
$\overline{M} \leq \max\{ M_1, M_2 \}$, one has
 \begin{align}
  J_2 & \leq  C(\eps) \int_0^t \int_\Omega |\nabla \varphi|^2 dx ds + \eps \int_0^t \int_\Omega \big(\max\{ M_1, M_2 \} \big)^a (M_1 - M_2)^2 dx ds.
  \label{Uniq:Eq.4.4}
  \end{align}
  At this point we go back to the equation \eqref{Uniq:Eq.4.1} where we consider the first integral on the right-hand side. 
  We claim that there exists a constant $D > 0$ such that
  \begin{align}
   D\big(  Q(M_1) - Q(M_2) \big)  (M_1 - M_2) \geq \big(\max\{ M_1, M_2 \} \big)^a (M_1 - M_2)^2.
  \label{Uniq:Eq.4.5}
  \end{align}
  In order to show \eqref{Uniq:Eq.4.5}, first we note that from \eqref{Def.q}, \eqref{def.Q} follows that
  \begin{align*}
  Q'(M) = p^2(M) \Big( \frac{M q(M)}{p(M)} \Big)' = \frac{M^a}{(1-M)^b}.
  \end{align*}
  Now, after integrating the previous expression from $0$ to $M$ we get
  \begin{align}
  Q(M) = \int_0^M \frac{s^a}{(1-s)^b} ds + Q(0).
  \label{Uniq:Eq.4.6}
  \end{align}
  Using \eqref{Uniq:Eq.4.6}, we calculate the term on the left-hand side of \eqref{Uniq:Eq.4.5}. In this way we have
  \begin{align}
  \big(  Q(M_1) - Q(M_2) \big) & (M_1 - M_2) = (M_1 - M_2) \int_{M_2}^{M_1} \frac{s^a}{(1-s)^b} ds \nonumber \\
   & \geq (M_1 - M_2) \int_{M_2}^{M_1} s^a ds 
    = \frac{(M_1 - M_2)^2}{a+1} M_1^a \frac{1 - (M_2/M_1)^{a+1}}{1 - (M_2/M_1)}.
   \label{Uniq:Eq.4.7}
  \end{align}
  A straightforward computation yields
  $$   \frac{1-x^{a+1}}{1-x} \geq C \max \{1,x^a\} = C \big(  \max \{ 1,x\} \big)^a , \qquad x \geq 0,\quad x \neq 1,   $$
  therefore \eqref{Uniq:Eq.4.7} leads to  
  \begin{align}
   \big(  Q(M_1) - Q(M_2) \big) & (M_1 - M_2)  \geq C(M_1 - M_2)^2 M_1^a \Big(  \max \Big\{1,\frac{M_2}{M_1}\Big\} \Big)^a \nonumber\\
   & = C(M_1 - M_2)^2 \big( \max\{ M_1, M_2 \} \big)^a,
    \label{Uniq:Eq.4.8}
  \end{align}
  proving in this way the claim \eqref{Uniq:Eq.4.5}.
Inserting estimate \eqref{Uniq:Eq.4.8} in \eqref{Uniq:Eq.4.4}, we get
\begin{align}
 J_2 \leq C(\eps) \int_0^t \int_\Omega & |\nabla \varphi|^2 dx ds \nonumber \\
   & + \eps D \int_0^t\int_\Omega   \big(  Q(M_1) - Q(M_2) \big) (M_1 - M_2) dx ds.
\label{Uniq:Eq.4.9}
\end{align}
 Now we go back to \eqref{Uniq:Eq.4.1}. Using the estimate \eqref{Uniq:Eq.4.9} in \eqref{Uniq:Eq.4.2} 
 we get
 \begin{align}
  \frac{1}{2} \|\nabla \varphi(t)\|_{L^2(\Omega)}^2 
  & \leq (\eps D - 1) \int_0^t \int_\Omega \big(  Q(M_1) - Q(M_2) \big)     (M_1 - M_2) dxds  \nonumber \\
  & +C(\eps) \int_0^t \int_\Omega  |\nabla \varphi|^2 dx ds .
 \label{Uniq:Eq.4.10}
 \end{align}
Let us take $\eps = 1/D$ in \eqref{Uniq:Eq.4.10}. By applying Gronwall's inequality 
and using the fact that $\nabla \varphi(0) = 0$, one gets $ \nabla \varphi(t) = 0$ for all $t>0$.
Finally, from $M_1(t) - M_2(t) = -\Delta \varphi(t) = 0$ it follows directly that $\forall t \in [0,T]$ one has $M_1(t) = M_2(t)$. In this way  the uniqueness of the total mass $M$ is proven. 

 In the second part of the proof, in order to prove the uniqueness of solution, we will follow the approach from \cite{ZaJu17} where the $E$-monotonicity technique of Gajewski \cite{Gaj94} has been applied.
This method is based on the convexity of the logarithmic entropy. For this purpose let us define the distance
\begin{align}
 \db(u,v)  = \sum_{i=1}^n \int_\Omega \Big( \xi(u_i) + \xi(v_i) - 2\xi\big( \frac{u_i + v_i}{2}  \big)  \Big) dx, \label{dist.1}
 \end{align}
 where $ \xi(s) = s \log s$. Notice that $\db(u,v) \geq 0$ due to the convexity of the function $\xi$.

Since $u_i$ and $v_i$ are only nonnegative and expressions like $\log u_i$, $\log v_i$ or $\log((u_i+v_i)/2$ may be undefined, we need the regularization of distance given by \eqref{dist.1}.
For that purpose let $0<\eps<1$. We introduce the regularized distance
\begin{align}
 \db_\eps(u,v)  = \sum_{i=1}^n \int_\Omega \Big( \xi_\eps(u_i) + \xi_\eps(v_i) - 2\xi_\eps\big( \frac{u_i + v_i}{2}  \big)  \Big) dx, \label{dist.eps.1}
 \end{align}
 where $\xi_\eps(s) = (s+\eps)  \log (s + \eps)$. 

Next, we observe that $\db_\eps(u(0),v(0)) = 0$ as $u$ and $v$ have the same initial data.
Using equation \eqref{cross.diff.sys.00} written  in the form
\begin{align}
 \partial_t u_i = \Div \Big(  pq \nabla u_i + u_i p^2 \nabla \big(  \frac{q}{p}\big) \Big) + r_i(u),
\label{cross.diff.sys.f1}
\end{align}
we compute
\begin{align*}
& \frac{d}{dt} \db_\eps(u,v)  =  \sum_{i=1}^n \Big( I_{1,i} +  I_{2,i} - I_{3,i} \Big),\quad I_{1,i} = \langle \partial_t u_i, \log(u_i + \eps) \rangle,\\
& I_{2,i} = \langle \partial_t v_i, \log(v_i + \eps) \rangle , \quad 
I_{3,i} = \Big\langle \partial_t(u_i+v_i), \log\Big( \frac{u_i+v_i}{2} +\eps \Big)  \Big\rangle .
\end{align*}
Taking into account the equation \eqref{cross.diff.sys.f1} and  performing 
partial integration gives
\begin{align*}
 I_{1,i} & = -  \int_\Omega  \frac{\nabla u_i}{u_i+\eps} \Big(  pq \nabla u_i + u_i p^2 \nabla \big(  \frac{q}{p}\big) \Big) dx + \int_\Omega r_i(u) \log(u_i + \eps) dx, \\
 I_{2,i} & = -  \int_\Omega  \frac{\nabla v_i}{v_i+\eps} \Big(  pq \nabla v_i + v_i p^2 \nabla \big(  \frac{q}{p}\big) \Big) dx + \int_\Omega r_i(v) \log(v_i + \eps) dx,\\
 I_{3,i} & = - \int_\Omega \frac{\nabla(u_i+v_i)}{u_i + v_i + 2\eps}\Big(  pq \nabla (u_i+ v_i) 
   + (u_i+v_i)  p^2 \nabla \big(  \frac{q}{p}\big) \Big) dx \\
     & \qquad+ \int_\Omega\big(  r_i(u)+r_i(v) \big)  
    \log \Big( \frac{u_i+v_i}{2} + \eps \Big) dx.
\end{align*}
By rearranging the terms we get 
\begin{align}\label{dt.dist}
&\frac{d}{dt} \db_\eps(u,v)  = \mathcal{F} + S + S_r,\\
\nonumber
&\mathcal{F} = -\sum_{i=1}^n 4 \int_\Omega \Big( |\nabla\sqrt{u_i+\eps}|^2 +  |\nabla\sqrt{v_i+\eps}|^2 - | \nabla\sqrt{u_i+v_i+2\eps}|^2  \Big) pq\, dx,\\
\nonumber
 & S =  - \sum_{i=1}^n \int_\Omega \Big( \frac{u_i}{u_i + \eps}  - \frac{u_i+v_i}{u_i+v_i+2\eps} \Big) p^2 \nabla u_i \cdot \nabla \big( \frac{q}{p}\big)\, dx \\ 
 \nonumber
 &\qquad  -  \sum_{i=1}^n \int_\Omega \Big(  \frac{v_i}{v_i + \eps} - \frac{u_i+v_i}{u_i+v_i+2\eps}  \Big) p^2 \nabla v_i \cdot \nabla \big( \frac{q}{p}\big) \,dx ,\\
 \nonumber
 & S_r  =  \sum_{i=1}^n \int_\Omega r_i(u) \log(u_i + \eps) dx +  \sum_{i=1}^n \int_\Omega r_i(v) \log(v_i + \eps) dx  \\
 \nonumber
  &\qquad -  \sum_{i=1}^n \int_\Omega \big(  r_i(u)+r_i(v) \big) \log \Big( \frac{u_i+v_i}{2} + \eps \Big) dx.
\end{align}
Using the fact that the Fisher information $\int_\Omega |\nabla u^{1/2}|^2 dx$ is subadditive (see \cite{ZaJu17}, Lemma 9), 
we get that $\mathcal F \leq 0$. Therefore, integrating \eqref{dt.dist} in time leads to
\begin{align}
 \db_\eps(u(t),v(t)) \leq \int_0^t S(s) ds + \int_0^t S_r(s) ds.
\label{uniq.1}
\end{align}
 Firstly, we treat the second 
  integral in \eqref{uniq.1}. 
We want to prove that
 \begin{align}
 \int_0^t S_r(s) ds \leq C\int_0^t \db_\eps(u(s),v(s))ds
 \label{uniq.6}
  \end{align}
 By taking into account the assumptions on the reaction terms given by \eqref{assum.react.0}--\eqref{assum.react.2} and the fact that $M_1=M_2$,
 the left-hand side of \eqref{uniq.6} can be rewritten as
\begin{align*}
& \int_0^t S_r(s) ds
  =\sum_{i=1}^n \int_0^t  \int_\Omega  r_i^{(0)}(M)\left( \log(u_i+\eps) + \log(v_i+\eps) - 2\log\left(\frac{u_i+v_i}{2}+\eps\right) \right)dx ds\\
  &+\sum_{i=1}^n \int_0^t  \int_\Omega  r^{(1)}(M)\left( u_i\log(u_i+\eps) + v_i\log(v_i+\eps) - (u_i+v_i)\log\left(\frac{u_i+v_i}{2}+\eps\right) \right)dx ds.
 \end{align*}
 Furthermore, from the definition of $\db_\eps(u,v)$ it follows
 \begin{align*}
  & \int_0^t S_r(s) ds\\
  &=\sum_{i=1}^n \int_0^t  \int_\Omega  (r_i^{(0)}(M)-\eps r^{(1)}(M))\left( \log(u_i+\eps) + \log(v_i+\eps) - 2\log\left(\frac{u_i+v_i}{2}+\eps\right) \right)dx ds\\
  &+\sum_{i=1}^n \int_0^t  \int_\Omega  r^{(1)}(M)\left( \xi_\eps(u_i) + \xi_\eps(v_i) - 2\xi_\eps\left( \frac{u_i + v_i}{2}  \right)  \right)dx ds\\
  &\leq\sum_{i=1}^n \int_0^t  \int_\Omega  (r_i^{(0)}(M)-\eps r^{(1)}(M))\left( \log(u_i+\eps) + \log(v_i+\eps) - 2\log\left(\frac{u_i+v_i}{2}+\eps\right) \right)dx ds\\
  &+C\sum_{i=1}^n \int_0^t  \db_\eps(u(s),v(s))ds ,
 \end{align*}
 where the last inequality comes from the fact that $r^{(1)}$ is upper bounded, which is a straightforward consequence of 
 \eqref{assum.react.0}--\eqref{assum.react.2}. 
 The convexity of $x\mapsto -\log(x)$ implies that 
 $$\log(u_i+\eps) + \log(v_i+\eps) - 2\log\left(\frac{u_i+v_i}{2}+\eps\right)\leq 0. $$
 Together with the assumptions on the reaction terms, we deduce that \eqref{uniq.6} holds.

From \eqref{uniq.1}, \eqref{uniq.6} it follows
  \begin{align*}
  \db_\eps(u(t),v(t)) & \leq 
   \int_0^t S(s) ds + C\sum_{i=1}^n \int_0^t  \db_\eps(u(s),v(s))ds. 
 \end{align*}
 Using Gronwall's Lemma yields
  \begin{align}
  \db_\eps(u(t),v(t)) & \leq e^{Ct }\sum_{i=1}^n \int_0^t \int_\Omega \Big| \frac{u_i}{u_i + \eps} - \frac{u_i+v_i}{u_i+v_i+2\eps} \Big| p^2 |\nabla u_i| \left|\nabla \left( \frac{q}{p}\right)\right| dxds \nonumber\\
  & + e^{Ct} \sum_{i=1}^n \int_0^t \int_\Omega \Big|  \frac{v_i}{v_i + \eps} - \frac{u_i+v_i}{u_i+v_i+2\eps}  \Big| p^2 |\nabla v_i| \left| \nabla \left( \frac{q}{p}\right)\right| \,dx ds .
 \label{uniq.1.1}
 \end{align}

Next, we want to apply the dominated convergence theorem to show that the right-hand side of \eqref{uniq.1.1} converges to zero as $\eps \to 0$. Let $g \in \{u_i,v_i\}$.
It is obvious that 
\begin{align*}
 \frac{g}{g + \eps} - \frac{u_i+v_i}{u_i+v_i+2\eps} &\to 0 \quad \textrm{ a.e. in $\{g>0\}$ as }\;  \eps \to 0,\\
 \big\vert \frac{g}{g + \eps} - \frac{u_i+v_i}{u_i+v_i+2\eps} \big\vert &\leq 2 \quad\mbox{a.e. in }\Omega\times (0,\infty).
\end{align*}
Therefore, in order to apply the dominated convergence theorem we need to show that
\[ p^2 \nabla g \cdot \nabla \Big( \frac{q}{p}\Big) \in L^1(\Omega \times (0,T)). \]
We make the calculation for $g = u_i$. For $g=v_i$ the calculation is completely the same.
The term of interest can be rewritten as follows
\begin{align}
  p^2 \nabla u_i \cdot \nabla \left( \frac{q}{p}\right) =  2 \sqrt{pq}\nabla \sqrt{u_i}\cdot \frac{\sqrt{u_i}}{\sqrt{pq}}
  p^2  \nabla \left( \frac{q}{p}\right).
\label{uniq.4}
\end{align}
Since $2 \sqrt{pq}\nabla \sqrt{u_i} \in L^2(\Omega \times (0,T))$ due to \eqref{bound.nau}, it remains to show that
$\frac{\sqrt{u_i}}{\sqrt{pq}}
  p^2  \nabla \big( \frac{q}{p}\big)  \in L^2(\Omega \times (0,T))$.
 For $u_i>0$ (and $v_i>0$ respectively), we have 
  \begin{align}
 \frac{\sqrt{u_i}}{\sqrt{pq}} p^2  \nabla \big( \frac{q}{p}\big) 
  = 2 \sqrt{u_i} p \nabla  \sqrt{\frac{q}{p}} 
  = 2p\nabla \sqrt{\frac{u_i q}{p}} - 2p \sqrt{\frac{q}{p}} \nabla \sqrt{u_i}. 
 \label{uniq.5}
 \end{align}
From Lemma~\ref{bound.nabla} and \eqref{entr.inequ.2} we conclude that
\[  \frac{\sqrt{u_i}}{\sqrt{pq}}
  p^2  \nabla \big( \frac{q}{p}\big) \in   L^2(\Omega \times (0,T)),  \]
 obtaining in this way finally that for $u_i, v_i > 0$ 
  \[  p^2 \nabla u_i \cdot \nabla \Big( \frac{q}{p}\Big)  \in   L^1(\Omega \times (0,T)).   \]
It remains to treat the case when $u_i = 0$ (or respectively $v_i = 0$). We want to show that
\[  p^2 \nabla u_i \cdot \nabla \Big( \frac{q}{p}\Big) = 0 \quad \textrm{ a.e. in the set } \{u_i = 0\}. \]
For this, we make the following estimate
 \begin{align*}
	    p^2|\nabla u_i|\left|\nabla\left(\frac{q}{p} \right)\right| 
\leq 4 \left(\left|\nabla \sqrt{u_i pq}\right| +\left|\sqrt{u_i}\nabla \sqrt{pq}\right|\right)\left(\left|p\nabla \sqrt{\frac{u_iq}{p}}\right| + \sqrt{pq}|\nabla \sqrt{u_i}|\right).
\end{align*}
    Now, due to \cite[p.153, 6.18 Corollary]{LL01}, it holds that 
      \begin{align*}
	    \nabla \sqrt{u_i pq} = 0 \; \textrm{ where } \; \sqrt{u_ipq}=0.
      \end{align*}
It follows that $\nabla \sqrt{u_i pq} = \sqrt{u_i}\nabla \sqrt{pq}=0$ (and a fortiori $p^2|\nabla u_i|\left|\nabla\left(q/p\right)\right|=0$) a.e. in the set $\{u_i=0\}$.
In this way, since $d_\eps$ is nonnegative, we get that $d_\eps(u(t),v(t)) \to 0$ as $\eps \to 0$, 
 which implies that
 \begin{align}
  \xi_\eps(u_i) + \xi_\eps(v_i) - 2\xi_\eps\Big( \frac{u_i+v_i}{2} \Big) \to 0\; \textrm{ as } \; \eps \to 0 \; \textrm{ a.e. in } \;
  \Omega \times(0,\infty).
 \label{uniq.2}
 \end{align}
 Same calculation like in \cite[p.26]{ZaJu17} gives the estimate
 \begin{align}
  \xi_\eps(u_i) + \xi_\eps(v_i) - 2\xi_\eps\Big( \frac{u_i+v_i}{2} \Big) \geq \frac{1}{8}(u_i-v_i)^2.
 \label{uniq.3}
 \end{align}
 Finally, \eqref{uniq.2} and \eqref{uniq.3} give that $u_i = v_i$ in $\Omega \times (0,\infty)$ 
 for $i=1,\ldots,n$. This concludes the proof of Theorem~\ref{Thm.Uniqueness}.
 $\hfill \Box$
 \begin{remark}\label{Remark.Neumann.uniq}
 \textnormal{The uniqueness result holds trivially (under the same assumptions) also
 if nonconstant Dirichlet data or homogeneous Neumann boundary conditions are considered.}
 \end{remark}

\section{Appendix}\label{sec.appendix}

\subsection{Formal derivation of the multi-species biofilm model from a spatially discrete lattice model}\label{Sec.Formal.der}
Here we discuss the modeling assumptions, which are strongly connected to the derivation of the multi-species biofilm model from a spatially discrete lattice ODE. More details on the derivation can be found in \cite{RSE16, ZaJu17, RaSuEb15}.
For simplicity, we sketch the derivation in 1D. Given a one-dimensional spatial lattice containing equidistant cells $x_j$ with cell distance 
$h=x_j-x_{j-1}$ of a finite interval, we consider the variables $u_i^j:=u_i(x_j)$, which model the density of the $i$th species at the $j$th grid cell. Moreover, transition rates $T_i^{j\pm}$ describe how species $u_i$ moves from cell $x_j$ to the neighboring cells $x_{j \pm 1}$. Biofilm movement into neighboring cells is driven by two principles: \textit{volume filling} and \textit{quenching} \cite{KHE09}. \textit{Volume filling} means that the movement depends on the available space in the local site, and since the site's capacity of accommodation mass is limited, we can normalize the population densities with respect to their maximum densities, which means that $u_i^j\leq 1$. Thus we can interpret $u_i^j$ as the volume fraction of site $j$ occupied by the species $u_i$. The discrete master equation, which describes the balance between the density of populations which leave the site to move into the neighboring sites, and the density of populations which arrive from neighboring sites, reads as 
  \begin{align}\label{master.equation}
    \partial_t u_i^j = T_i^{(j-1)+}u_i^{j-1} + T_i^{(j+1)-}u_i^{j+1} - \left(T_i^{j+} + T_i^{j-}\right) + r_i^j
  \end{align}
where $r_i^j=r_i^j(u^j)$ is the net growth rate of the $i$th species, and the transition rates have the form
  \begin{align*}
    T_i^{j\pm} = \alpha_i q_i(u_1^j, \ldots, u_n^j)p_i(u_1^{j\pm 1}, \ldots, u_n^{j\pm 1}),
  \end{align*}
where $\alpha_i=\alpha_i(h)$ measure how fast populations move between neighboring cells, and the nonnegative continuous transition functions $q_i$ and $p_i$ describe the local movement of the species from one cell to the other. The transition function $q_i(u)$ measures the incentive of the density of species $i$ at grid cell $x_j$ to leave the cell $x_j$, and $p_i(u_1^{j\pm 1}, \ldots, u_n^{j\pm 1})$ models the attractivity of the cell population $u_i^j$ for the incoming individuals $u_i^{j\pm1}$ from the neighboring sites $j\pm 1$. The second principle for biofilm movement is that as long as there is capacity to accommodate new biomass locally in that cell, the incentive to move to a neighboring cell is small, which is called \textit{quenching}. The transition from the spatially discretized to the continuous model is now performed in a formal diffusive limit. First, we interpolate the grid functions by setting 
$$u_i(t,x)=u_i^j(t), \quad  \mbox{for} \quad x_j\leq x \leq x_{j+1}.$$
Now, assuming sufficient smoothness of the functions $u_1,\ldots, u_n$, we can approximate $u_i(t, x^{j\pm1})$, $q(u(t,x^{j\pm 1}))$, $p(u(t,x^{j\pm 1}))$ by second order Taylor polynomials around $u(t,x_j)$. By substituting all these expressions into the master equation \eqref{master.equation} and performing the (formal) diffusive limit $h \to 0$ under the assumption that $\lim_{h \to 0}\alpha_i h^2 = \alpha_{i0}>0$, we get for $u=(u_1, \ldots, u_n)$ the equation
\begin{align*}
  \pa_t u_i = \alpha_{i0} \frac{\pa}{\pa x} \left(\sum_{j=1}^n A_{ij}(u)\frac{\pa u_j}{\pa x}\right) + r_i(u), \quad i=1,\ldots,n,
\end{align*}
 where the diffusion coefficients have the form
 \begin{align}\label{tilde.A}
 A_{ij}(u)= \delta_{ij} p_i(u)q_i(u) + u_i \left(p_i(u)\frac{\pa q_i}{\pa {u_j}}(u) -  q_i(u)\frac{\pa p_i}{\pa {u_j}}(u)\right).
 \end{align}
 In more than one dimension, the same procedure can be applied, leading to the system
 \begin{align*}
   \pa_t u_i = \alpha_{i0} \Div\left(\sum_{j=1}^n A_{ij}(u)\nabla u_j\right) + r_i(u), \quad i=1,\ldots,n,
 \end{align*}
where $A$ is defined in \eqref{tilde.A}.

\subsection{Numerical illustration of the relative entropy}\label{Sec.Num}
In this subsection we present numerical simulations of the relative entropy given by \eqref{H} with the respect to time. 
 We take a three-species model \eqref{cross.diff.sys.00} on the rectangular domain $\Omega = [0,1]\times[0,1]$.
The solution is calculated using the Distributed and Unified Numerics Environment is a modular toolbox for solving partial differential equations (PDEs) with grid-based methods DUNE  \cite{DUNE_Publ}.  For calculating the solution we use standard FEM on the rectangular grid with $Q_1$-elements. More precisely, we used the following Dune modules:
1. Core modules (dune-grid, dune-geometry, dune-localfunctions, dune-common, dune-istl); 2. Discretization modules (dune-fem, dune-pdelab).
The initial conditions are:
\begin{align*}
u_1(x,y;0) = \begin{cases}
            u_{D,1} + \varepsilon,\quad 0.2 \leq x \leq  0.5 \; \textrm{ and }\; 0 \leq y \leq  0.2,\\
           u_{D,1}, \quad \textrm{otherwise},
  \end{cases}
\end{align*}
\begin{align*}
u_2(x,y;0) = \begin{cases}
           u_{D,2} + \varepsilon,\quad 0.5 \leq x \leq  0.8 \; \textrm{ and }\; 0 \leq y \leq  0.2,\\
           u_{D,2}, \quad \textrm{otherwise},
  \end{cases}
\end{align*}
\begin{align*}
u_3(x,y;0) = \begin{cases}
           u_{D,3} + \varepsilon,\quad 0.2 \leq x \leq  0.8 \; \textrm{ and }\; 0 \leq y \leq  0.2,\\
            u_{D,3}, \quad \textrm{otherwise},
  \end{cases}
\end{align*}
Here, $u_{D,i} = \varepsilon$, $i=1,2,3$ where $\varepsilon = 0.1$, the time-step is $dt = 10^{-4}$ and the final time $T_{\textrm{fin}} = 10$.
For the function $p(M)$ we chose:
\begin{align}
p(M) = \exp\Big({-\frac{1}{(1-M)^\kappa}}\Big), \quad M = u_1 + u_2 + u_3.
\label{p(s)}
\end{align}

Concerning different choices of parameters $a$, $b$ from \eqref{Eq.1.0} and $\kappa$ in \eqref{p(s)} as well as different boundary conditions (mixed Dirichlet-Neumann or homogeneous Neumann), we performed the following two tests.\smallskip\\
{\bf Test 1.} With this test we wanted to illustrate numerically our analytical result given by Theorem~\ref{Thm.LTB}. For that purpose here we considered our model with mixed Dirichlet-Neumann boundary conditions. More precisely, we took the Dirichlet boundary conditions $u_{D_i}$, $i=1,2,3$ on the upper-side of the rectangle and homogeneous Neumann on other three sides.
 For parameters we took $a = 2$, $b=2$, and $ \kappa = 1$ and we used $r_i^D = u_{D,i} - u_i$ as the reaction terms.
We note that Figure~\ref{FigNumSim} (left) shows the exponential convergence of the solution to the steady state, which we were not able to obtain with our analytical tools. \smallskip \\
{\bf Test 2.} This test corresponds to Remark~\ref{Remark.Neumann.LTB}. Here we considered the homogeneous Neumann boundary conditions without any reaction terms. For parameters we took $a = 1$, $b=2$ and  $\kappa = 0.9$.
In Figure~\ref{FigNumSim} (right), we observe very fast stabilization of the solution to the constant steady-state.
\begin{figure}\label{FigNumSim}
\includegraphics[width=7cm]{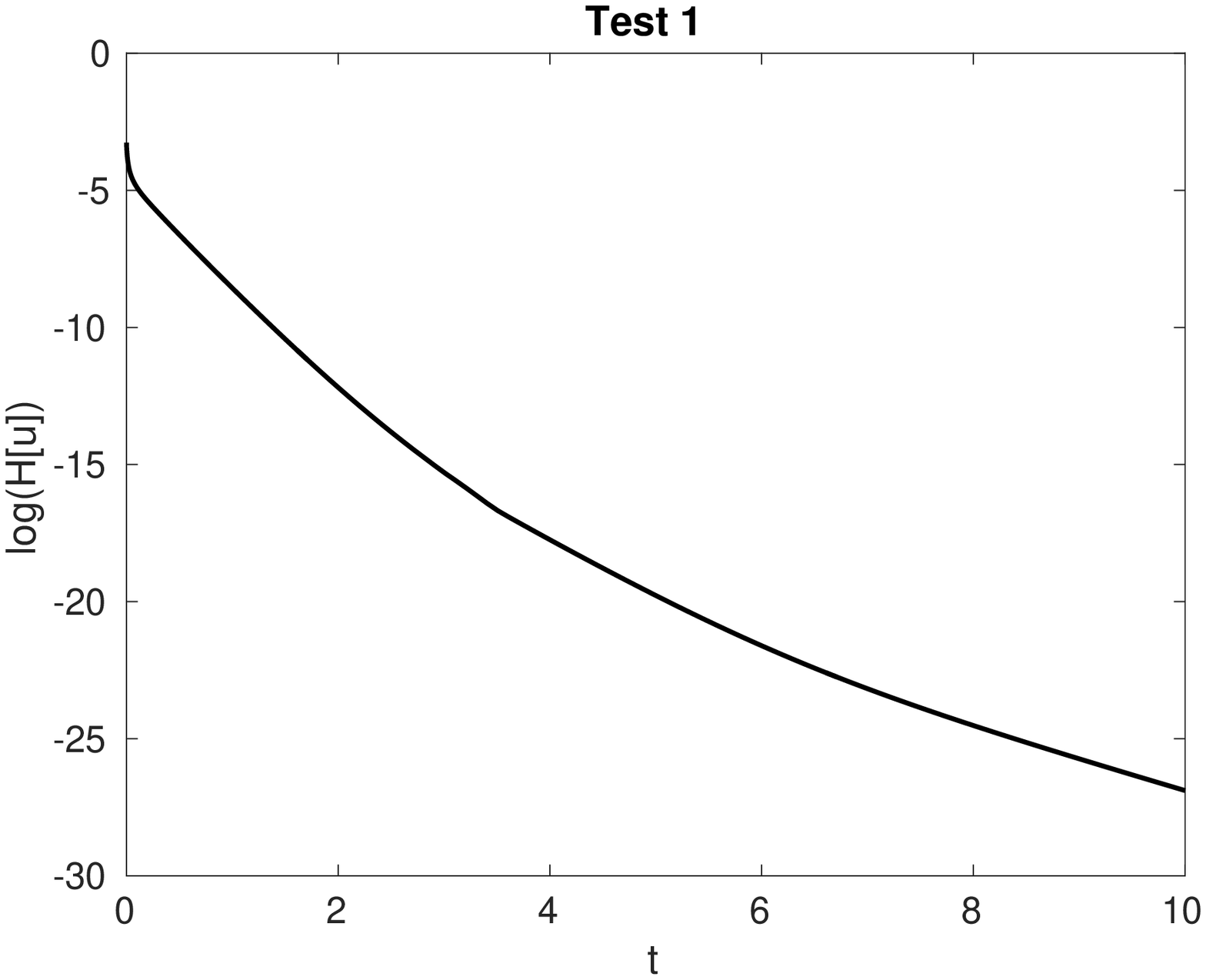}
\includegraphics[width=7cm]{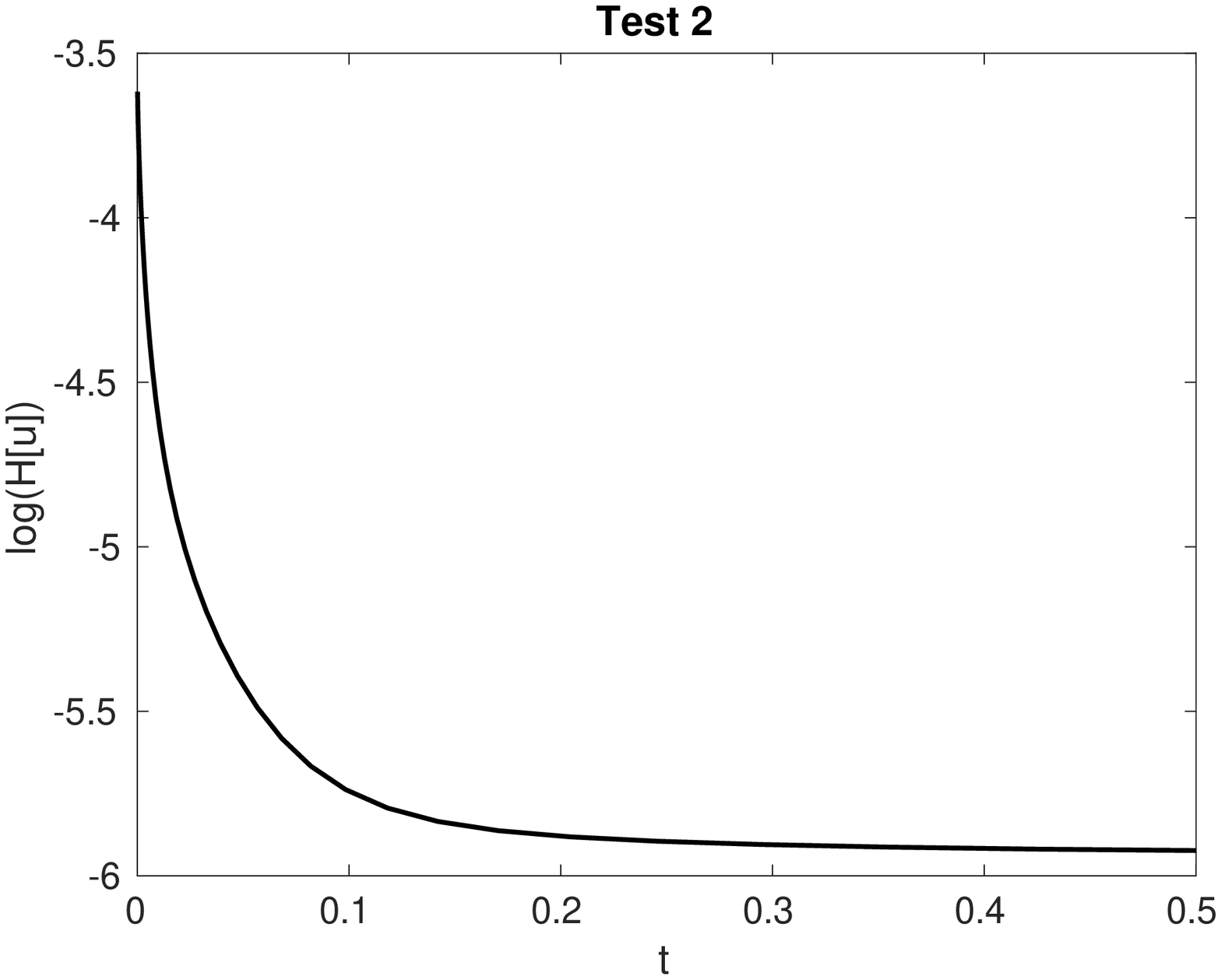}
\caption{Logarithm of the relative entropy vs time}
\end{figure}

Summarizing up, it seems that in the situation described by Theorem~\ref{Thm.LTB} the convergence to the steady state is actually exponential.
The algebraic decay result proved in the Theorem might be not optimal, most likely due to limitations in the analytical methods employed in the proof. 
Furthermore, convergence to the steady state appears to
be helped by a reaction term with a suitable (dissipative) structure (as it is to be expected): the rate of convergence is higher in Test 1 than in Test 2,
and the relative entropy reaches much smaller values in Test 1 than in Test 2. In fact, the relative entropy seems to stabilize around $10^{-6}$ in Test 2, which might be a symptom of 
numerical instability.

\subsection{Additional auxiliary result}\label{appendix: aux.res} 
\begin{lemma}[Variant of the Poincar\'e Inequality]\label{lem.Poi.sing}
 Let $\Omega\subset\R^d$ bounded open and connected. Let $\lambda\in (0,1)$, $\alpha>0$. There exists a constant $C>0$
 such that 
 $$  \|(1-M)^{-\alpha}\|_{L^2(\Omega)}\leq C\left(1+\|\na (1-M)^{-\alpha}\|_{L^2(\Omega)}\right ) $$
 for every function $M\in H^1(\Omega)$ such that $0\leq M < 1$ a.e. in $\Omega$ and $|\Omega|^{-1}\int_\Omega M dx \leq \lambda$.
\end{lemma}
\begin{proof}
 By contradiction. Assume $\forall k\geq 1$ there exists $M_k\in H^1(\Omega)$ such that 
 $0\leq M_k < 1$ a.e. in $\Omega$, $|\Omega|^{-1}\int_\Omega M_k dx \leq \lambda$, and
 $$  \|(1-M_k)^{-\alpha}\|_{L^2(\Omega)}> k\left(1+\|\na (1-M_k)^{-\alpha}\|_{L^2(\Omega)}\right ) $$
 for all $k\geq 1$.
 Let us define $f_k = \frac{(1-M_k)^{-\alpha}}{\|(1-M_k)^{-\alpha}\|_{L^2(\Omega)}}$. It follows that $\|f_k\|_{L^2(\Omega)}=1$,
 $$ \|\na f_k\|_{L^2(\Omega)}<\frac{1}{k}-\frac{1}{\|(1-M_k)^{-\alpha}\|_{L^2(\Omega)}} . $$
 Clearly $\|(1-M_k)^{-\alpha}\|_{L^2(\Omega)}\to\infty$ as $k\to\infty$, which in particular implies that
 $\|\na f_k\|_{L^2(\Omega)}\to 0$ as $k\to\infty$. Moreover $f_k$ is bounded in $H^1(\Omega)$, which by compact Sobolev embedding implies
 that, up to subsequences, $f_k\to f$ strongly in $L^2(\Omega)$. Since $\|\na f_k\|_{L^2(\Omega)}\to 0$ as $k\to\infty$ it follows that
 $$
 \int_\Omega f\pa_{x_i}\phi dx = \lim_{k\to\infty}\int_\Omega f_k\pa_{x_i}\phi dx = 
 -\lim_{k\to\infty}\int_\Omega \pa_{x_i}f_k\,\phi dx = 0\qquad\forall \phi\in H^1_0(\Omega).
 $$
 This means that $f$ is constant. Moreover, the fact that $\|f_k\|_{L^2(\Omega)}=1$ for all $k$, together with the strong convergence of $f_k$,
 implies that $f>0$. However, up to subsequences, $f_k\to f$ a.e. in $\Omega$, i.e.
 $$
 \frac{(1-M_k)^{-\alpha}}{\|(1-M_k)^{-\alpha}\|_{L^2(\Omega)}}\to f\qquad \mbox{a.e. in }\Omega.
 $$
 Since $\|(1-M_k)^{-\alpha}\|_{L^2(\Omega)}\to\infty$ as $k\to\infty$, we deduce that $M_k\to 1$ a.e. in $\Omega$, and therefore (by dominated
 convergence) also strongly in $L^1(\Omega)$. As a consequence $|\Omega|^{-1}\int_\Omega M_k dx\to 1$. This is a contradiction to the fact that
 $|\Omega|^{-1}\int_\Omega M_k dx \leq \lambda<1$ for all $k\geq 1$. This finishes the proof.
\end{proof}


\end{document}